\documentclass[12pt]{amsart}
\usepackage{amssymb}

\begin{document}
\numberwithin{equation}{section}

\def\Label#1{\label{#1}}

\newtheorem{Claim}{Claim}

\def\1#1{\ov{#1}}
\def\2#1{\widetilde{#1}}
\def\3#1{\mathcal{#1}}
\def\4#1{\widehat{#1}}

\def\s{s}
\def\k{\kappa}
\def\ov{\overline}
\def\span{\text{\rm span}}
\def\tr{\text{\rm tr}}
\def\GL{{\sf GL}}
\def\xo {{x_0}}
\def\Rk{\text{\rm Rk\,}}
\def\sg{\sigma}
\def \emxy{E_{(M,M')}(X,Y)}
\def \semxy{\scrE_{(M,M')}(X,Y)}
\def \jkxy {J^k(X,Y)}
\def \gkxy {G^k(X,Y)}
\def \exy {E(X,Y)}
\def \sexy{\scrE(X,Y)}
\def \hn {holomorphically nondegenerate}
\def\hyp{hypersurface}
\def\prt#1{{\partial \over\partial #1}}
\def\det{{\text{\rm det\,}}}
\def\wob{{w\over B(z)}}
\def\co{\chi_1}
\def\po{p_0}
\def\fb {\bar f}
\def\gb {\bar g}
\def\Fb {\ov F}
\def\Gb {\ov G}
\def\Hb {\ov H}
\def\zb {\bar z}
\def\wb {\bar w}
\def \qb {\bar Q}
\def \t {\tau}
\def\z{\chi}
\def\w{\tau}
\def\Z{\zeta}

\def \T {\theta}
\def \Th {\Theta}
\def \L {\Lambda}
\def\b{\beta}
\def\a{\alpha}
\def\o{\omega}
\def\l{\lambda}
\def \CFT{\text{\rm CFT}}

\def \im{\text{\rm Im }}
\def \re{\text{\rm Re }}
\def \Char{\text{\rm Char }}
\def \supp{\text{\rm supp }}
\def \codim{\text{\rm codim }}
\def \Ht{\text{\rm ht }}
\def \Dt{\text{\rm dt }}
\def \hO{\widehat{\mathcal O}}
\def \cl{\text{\rm cl }}
\def \bR{\mathbb R}
\def \bC{\mathbb C}
\def \bP{\mathbb P}
\def \bS{\mathbb S}
\def \C{\mathbb C}
\def \bL{\mathbb L}
\def \bZ{\mathbb Z}
\def \bN{\mathbb N}
\def \scrF{\mathcal F}
\def \scrK{\mathcal K}
\def \scrM{\mathcal M}
\def \cR{\mathcal R}
\def \scrJ{\mathcal J}
\def \scrA{\mathcal A}
\def \scrO{\mathcal O}
\def \scrV{\mathcal V}
\def \scrL{\mathcal L}
\def \scrE{\mathcal E}
\def \hol{\text{\rm hol}}
\def \aut{\text{\rm aut}}
\def \Aut{\text{\rm Aut}}
\def \J{\text{\rm Jac}}
\def\jet#1#2{J^{#1}_{#2}}
\def\gp#1{G^{#1}}
\def\gpo{\gp {2k_0}_0}
\def\emmp {\scrF(M,p;M',p')}
\def\rk{\text{\rm rk}}
\def\Orb{\text{\rm Orb\,}}
\def\Exp{\text{\rm Exp\,}}
\def\ess{\text{\rm Ess\,}}
\def\mult{\text{\rm mult\,}}
\def\Jac{\text{\rm Jac\,}}
\def\Span{\text{\rm span\,}}
\def\d{\partial}
\def\D{\3J}
\def\pr{{\rm pr}}
\def\dbl{[\![}
\def\dbr{]\!]}
\def\nl{|\!|}
\def\nr{|\!|}

\def \D{\text{\rm Der}\,}
\def \Rk{\text{\rm Rk}\,}
\def \ima{\text{\rm im}\,}
\def \vfi{\varphi}

\title[]
{Partial rigidity of degenerate CR embeddings into spheres}
\author [P. Ebenfelt]
{Peter Ebenfelt}

\address{P. Ebenfelt: Department of Mathematics, University of
California at San Diego, La Jolla, CA 92093-0112, USA}
\email{pebenfel@math.ucsd.edu}


\date{\number\year-\number\month-\number\day}

\abstract In this paper, we study degenerate CR embeddings $f$ of a strictly pseudoconvex hypersurface $M\subset \bC^{n+1}$ into a sphere $\bS$ in a higher dimensional complex space $\bC^{N+1}$. The degeneracy of the mapping $f$ will be characterized in terms of the ranks of the CR second fundamental form and its covariant derivatives. In 2004, the author, together with X. Huang and D. Zaitsev, established a rigidity result for CR embeddings $f$ into spheres in low codimensions. A key step in the proof of this result was to show that degenerate mappings are necessarily contained a complex plane section of the target sphere (partial rigidity). In the 2004 paper, it was shown that if the total rank $d$ of the second fundamental form and all of its covariant derivatives is $<n$ (here, $n$ is the CR dimension of $M$), then $f(M)$ is contained in a complex plane of dimension $n+d+1$. The converse of this statement is also true, as is easy to see. When the rank $d$ exceeds $n$, it is no longer true, in general, that $f(M)$ is contained in a complex plane of dimension $n+d+1$, as can be seen by examples. In this paper, we carry out a systematic study of degenerate CR mappings into spheres. We show that when the ranks of the second fundamental form and its covariant derivatives exceed the CR dimension $n$, then partial rigidity may still persist, but there is a "defect" $k$ that arises from the ranks exceeding $n$ such that $f(M)$ is only contained in a complex plane of dimension $n+d+k+1$. Moreover, this defect occurs in general, as is illustrated by examples.
\endabstract


\thanks{The author is supported in part by
DMS-1001322.}

\newtheorem{Thm}{Theorem}[section]
\newtheorem{Def}[Thm]{Definition}
\newtheorem{Cor}[Thm]{Corollary}
\newtheorem{Pro}[Thm]{Proposition}
\newtheorem{Lem}[Thm]{Lemma}
\newtheorem{Rem}[Thm]{Remark}
\newtheorem{Ex}[Thm]{Example}
\newtheorem{Con}[Thm]{Conjecture}

\maketitle

\section{Introduction}

In the theory of Levi nondegenerate CR manifolds of hypersurface type, as developed by E. Cartan (\cite{Cartan32}, \cite{Cartan33}), N. Tanaka (\cite{Tanaka62}, \cite{Tanaka75}), and S.-S. Chern and J. Moser (\cite{CM74}), the role of flat models is played by the nodegenerate hyperquadrics. In the strictly pseudoconvex case, the flat model is the unit sphere $\bS=\bS^N\subset \bC^{N+1}$ (here, we use the superscript $N$ to denote the CR dimension of a real hypersurface; the real dimension of $\bS^N$ is $2N+1$), defined by
$$
\sum_{k=1}^{N+1}|z_k|^k=1.
$$
An important problem that has attracted much attention since the work of Chern and Moser is to understand the CR submanifold structure of the hyperquadrics, especially the sphere. By the work of Webster, Faran, and others (\cite{Webster79}, \cite{Faran86}, \cite{Forstneric86}, \cite{Huang99}), it is known that a local CR embedding of a piece of the sphere $M^n\subset \bS^n\subset \bC^{n+1}$ into $\bS^N\subset\bC^{N+1}$ is necessarily linear, i.e.\ contained in an $(n+1)$-plane section of $\bS^N$, {\it provided} that $N-n<n$. Consequently, any such local CR embedding is equal to the standard linear embedding of the sphere into an $(n+1)$-subspace section of $\bS^N$ composed with an automorphism of $\bS^N$.  (The space of automorphisms of $\bS^N$ is well understood; it consists of automorphisms of the projective space $\bP^{N+1}$ that preserve the Hermitian form associated with the sphere as a real hypersurface in $\bP^{N+1}$). This result was extended by the author, X. Huang, and D. Zaitsev \cite{EHZ04}, \cite{EHZ05} to the case of more general strictly pseudoconvex hypersurfaces $M^n\subset \bC^{n+1}$ with low CR complexity $\mu(M)$ (see also \cite{Webster79} for an earlier result along these lines); the CR complexity $\mu(M)$ is defined as the minimal codimension $N_0-n$ (possibly $\infty$) such that there exists a local CR embedding $f_0\colon M^n\to \bS^{N_0}$. The results in \cite{EHZ04}, \cite{EHZ05} imply that {\it rigidity} holds for local embeddings $f\colon M^n\to \bS^N$, i.e.\ any such $f$ is of the form $f=T\circ L\circ f_0$ where $L$ denotes the standard embedding of $\bS^{N_0}$ into a subspace section of $\bS^{N}$ and $T$ is an automorphism of $\bS^N$, {\it provided} that $N-n+\mu(M)<n$. Currently, much effort is dedicated to classifying local embeddings into spheres (up to the equivalence induced by compositions with sphere automorphisms) beyond the rigidity range for the codimension. We mention here in particular the works \cite{HuangJi01}, \cite{HuangJiXu06}, \cite{HuangJiYin12} that deal with what is now known as the Huang-Ji-Yin Gap Conjecture \cite{HuangJiYin09} for mappings $\bS^n\to\bS^N$. The Gap Conjecture predicts that, for a fixed $n\geq 2$, all local CR embeddings $f\colon \bS^n\to \bS^N$ for codimensions $N-n$ in a finite number of finite ranges $G_1,G_2,\ldots,G_k \subset \bZ_+$ (the "gaps"), are necessarily contained in proper complex plane sections of the target sphere. We refer to this as {\it partial rigidity}. (The dimensions of the complex plane sections are also predicted by the conjecture.) The first gap $G_1=\{1,\ldots,n-1\}$ corresponds to the rigidity result by Faran mentioned above. The second and third gap, $G_2=\{n+2,\ldots, 2n-2\}$ and $G_3=\{2n+3,\ldots, 3n-4\}$, were established in \cite{HuangJiXu06} and \cite{HuangJiYin12}, respectively. The reader is referred to \cite{HuangJiYin09},\cite{HuangJiYin12} for a description of the predicted gaps beyond the third one. For related works on classification and (partial) rigidity of CR mappings into hyperquadrics, the reader is also referred to e.g. \cite{DAngelo88}, \cite{Dangelo91}, \cite{Huang99}, \cite{BH05}, \cite{JPDLeblPeters07}, \cite{BEH08}, \cite{BEH09}, \cite{ESh10}, \cite{JPDLebl11} and references therein.

An important step in many of the works mentioned above is to establish a preliminary classification of CR mappings into spheres according to their degeneracies and to prove partial rigidity (as described above) for degenerate maps. In this paper, we shall develop a general framework for such a classification in an "intermediate" range of codimensions. To explain this more precisely and formulate our main result, we need to recall some definitions and introduce notation; we shall follow the notation and conventions in \cite{EHZ04}. We shall consider smooth CR maps $f\colon M=M^n\subset \bC^{n+1}\to \hat M=\hat M^N\subset \bC^{N+1}$, where $M=M^n$ is a strictly pseudoconvex, connected hypersurface in $\bC^{n+1}$ (and thus the superscript $n$, as mentioned above, denotes the CR dimension of $M$) and $\hat M=\hat M^N$ is a strictly pseudoconvex hypersurface in $\bC^{N+1}$; our main result will concern the situation where the target $\hat M^N$ is the sphere $\bS^N$, but the definitions and notation below are valid in the general situation. We first note, as is well known, that any non-constant CR mapping $f\colon M^n\to \hat M^N$ is (due to the strict pseudoconvexity of the source and target) a transversal CR embedding locally (i.e.\ $f$ is a transversal CR immersion). The degeneracy of a local CR embedding $f\colon M^n\to \hat M^N$ will be measured by its CR second fundamental form and covariant derivatives. We will first introduce these concepts in an elementary (extrinsic) way in order for us to quickly be able to formulate our main result. Let $p_0\in M$ and let $\hat \rho(Z,\bar Z)$ be a local defining function for $\hat M\subset \bC^{N+1}$ near $\hat p_0:=f(p_0)\in \hat M$. Let $L_1,\ldots, L_n$ be a local basis for the $(1,0)$-vector fields on $M$ near $p_0$. We shall denote by $T^{''}_{\hat p} \bC^{N+1}$ the space of $(0,1)$-covectors in $\bC^{N+1}$ at $\hat p$ and identify this space with $\bC^{N+1}$ using the basis $d\bar Z_1,\ldots,d\bar Z_{N+1}$. Using this identification, we shall denote by $\hat\rho_{\bar Z}=\bar\partial\hat\rho$. We shall also identify $T^{''}_{\hat p} \bC^{N+1}$ with $T^{''}_{\hat p} \hat M$ in the standard way. Following Lamel \cite{Lam01}, we introduce an increasing sequence of subspaces
\begin{equation}\Label{hatEk}
\hat E_1(p)\subset \hat E_2(p)\subset\ldots \hat E_l(p)\subset \ldots \subset T{''}_{\hat p} \hat M,
\end{equation}
where $\hat E_l(p)$ is defined by
\begin{equation}\Label {hatEkdef}
\hat E_l(p):=\span_\bC\left\{L^{J}(\hat\rho_{\bar Z}\circ f)\colon J\in\bZ_+^n,\ |J|\leq l\right\}.
\end{equation}
In \eqref{hatEkdef}, we have used multi-index notation $L^J:=L_1^{J_1}\ldots L_n^{J_n}$ and $|J|:=J_1+\ldots +J_n$. For a transversal local embedding $f$ (i.e.\ for any non-constant map $f$ in the strictly pseudoconvex situation), the space $\hat E_1(p)$ has dimension $n+1$. We now introduce a sequence of dimensions $0\leq d_2(p)\leq d_3(p) \leq\ldots\leq d_l(p)\leq \ldots$ defined as follows
\begin{equation}\Label{dkdef}
d_l(p):=\dim_\bC \hat E_l(p)/\hat E_1(p).
\end{equation}
We can of course also define $d_1(p)$ in this way, resulting in $d_1(p)=0$; the integer $d_1(p)$ clearly carries no information, but we shall sometimes use $d_1=0$ to simplify the notation.
We note that on an open and dense subset of $M$ these dimensions will be locally constant and on each component $U\subset M$ of this set there will be an integer $l_0$ such that 
\begin{equation}\Label{dkconst}
d_2(p)<d_3(p)<\ldots <d_{l_0}(p)=d_{l_0+1}(p)=\ldots
\end{equation}
The dimensions $d_2(p), d_{3}(p),\ldots$ can be interpreted as ranks of the CR second fundamental form of $f$ and its covariant derivatives, as will be explained in Section \ref{prelim} below. 

We now return to the special case where $\hat M^N$ is the sphere $\bS^N$. Our main result can be stated as follows:

\begin{Thm}\Label{Main0} Let $M\subset \bC^{n+1}$ be a strictly pseudoconvex, connected hypersurface, and $f\colon M\to  \bS^N\subset \bC^{N+1}$ a smooth CR mapping. Let the dimensions $d_l(p)$ be given by \eqref{dkdef}, and let $U$ be an open subset of $M$ such that $d_l=d_l(p)$, for $2\leq l\leq l_0$, are constant on $U$ and \eqref{dkconst} holds. Assume that there are integers $0\leq k_2, k_3,\ldots,k_{l_0}\leq n-1$, such that:
\begin{equation}
\begin{aligned}\Label{conds0}
d_l-d_{l-1} <&\sum_{j=0}^{k_l}(n-j),\quad l=2,\ldots, l_0,\quad (d_1=0)\\
k:=&\sum_{l=2}^{l_0}k_l<n.
\end{aligned}
\end{equation}
Then $f(M)$ is contained in a complex plane of dimension $n+d+k+1$, where $k$ is defined in \eqref{conds0} and $d:=d_{l_0}$.
\end{Thm}

In the special case where $d<n$ (which clearly allows us to choose all the $k_l$ to be 0 and, hence, $k=0$), Theorem \ref{Main0} asserts that $f(M)$ is contained in a complex plane of dimension $n+d+1$. This is precisely Theorem 2.2 in \cite{EHZ04} (although the reader is cautioned that the notation in that paper is different), which was a crucial ingredient in the proof of the rigidity result in that paper. (A related result was proved in the more general setting of Levi nondegenerate (but not necessarily pseudoconvex) hypersurfaces in \cite{ESh10} and used there to prove partial rigidity for mappings into hyperquadrics.) We note, as is easy to verify, that if $f\colon M\to \bS^N$ is a CR mapping such that $f(M)$ is contained in a proper complex plane of dimension $n+r+1$, then necessarily $d\leq r$. Theorem 2.2 in \cite{EHZ04} provides the converse of this statement when $d<n$. Our main result here, Theorem \ref{Main0}, offers a "converse with a penalty" for situations where $d$ exceeds $n$ and where the "penalty" is quantified by the integer $k$. The penalty $k$ actually occurs in general, as can be seen in examples. We illustrate this fact with the following two examples:

\begin{Ex} {\rm Consider the following family of holomorphic mappings $D_t\colon \bC^{n+1}\to \bC^{N+1}$, with $N=2n+1$ and $t\in \bR$, sending $\bS^n$ into $\bS^N$:
\begin{equation}\Label{JPDt}
D_t(z,\ldots,z_{n+1}):=(z_1,\ldots,z_n,(\cos t)z_{n+1},(\sin t)z_1z_{n+1},\ldots,(\sin t)z_{n+1}^2).
\end{equation}
A straightforward calculation shows that, for $t\in (0,\pi/2)$, on a dense open set of $\bS^n$ we have $l_0=2$ and $d=d_2=n$. Thus, we have $k=k_2=1$ and Theorem \ref{Main0} "predicts" that $D_t$ will be contained in a complex plane of dimension $n+d+k+1=n+n+1+1=2n+2$, which of course is not much of a "prediction" since this is also the dimension of the target space $\bC^{N+1}$ in this case. However, it is also easily seen that $D_t$, for $t\in (0,\pi/2)$, is not contained in any proper complex plane, which shows that the conclusion of Theorem \ref{Main0} cannot be improved to yield a plane of dimension $n+d+1$ in this example. The family of mappings in \eqref{JPDt} was discovered by D'Angelo \cite{DAngelo88}, who also showed that $D_{t_1}$ and $D_{t_2}$ are inequivalent (in the sense described above) for $t_1\neq t_2$.
}
\end{Ex}

\begin{Ex} {\rm Now, consider the following family of holomorphic mappings $H_{s,t}\colon \bC^{n+1}\to \bC^{N+1}$, with $N=3n+2$ and $s,t\in \bR$, sending $\bS^n$ into $\bS^N$:
\begin{multline}\Label{Huangst}
H_{s,t}(z,\ldots,z_{n+1}):=(z_1,\ldots,z_{n-1},(\cos s)z_n,z_{n+1},(\sin s)z_1z_{n},\ldots,\\(\sin s)z_{n}^2, (\sin s\cos t)z_nz_{n+1},(\sin s\sin t) z_1z_nz_{n+1},\ldots, (\sin s\sin t)z_nz_{n+1}^2).
\end{multline}
As above, a straightforward calculation shows, for $s,t\in (0,\pi/2)$, that on a dense open set of $\bS^n$ we have $l_0=3$, $d_2=d_3=n$, and hence $d=d_2+d_3=2n$. Thus, we have $k_2=k_3=1$, $k=k_2+k_3=2$ and Theorem \ref{Main0} "predicts" that $H_{s,t}$ will be contained in a complex plane of dimension $n+d+k+1=n+2n+2+1=3n+3$, which again is also the dimension of the target space $\bC^{N+1}$. As above, it is easily seen that $H_{s,t}$, for $s,t\in (0,\pi/2)$, is not contained in any proper complex plane, showing that the conclusion of Theorem \ref{Main0} is sharp also in this case.
The mapping $H_{s,t}$ and further examples of this type were given in in \cite{HuangJiYin09}.
}
\end{Ex}

We conclude this introduction with the following remark, which provides a situation where the assumptions in \eqref{conds0} are automatically satisfied.

\begin{Rem} {\rm We note that if the codimension satisfies
$$
N-n<\frac{n(n+1)}{2},
$$
then necessarily there are $0\leq k_2,k_3,\ldots, k_{l_0}\leq n-1$ such that the two conditions in \eqref{conds0} are satisfied. Indeed, note that
$$
\sum_{j=0}^{n-1} (n-j)=\frac{n(n+1)}{2},
$$
and, hence, the existence of $k_l$ such that the first condition is satisfied follows immediately (since $d\leq N-n<n(n+1)/2$). Let us choose the $k_l$ minimal, i.e.\ such that
$$
d_l-d_{l-1}\geq \sum_{j=0}^{k_l-1}(n-j),
$$
where $d_1$ and the sum when $k_l=0$ are understood to be 0. By telescoping $d=d_{l_0}$, we conclude that
\begin{equation}\Label{sum10}
d\geq\sum_{l=2}^{l_0}\sum_{j=0}^{k_l-1}(n-j).
\end{equation}
A moments reflection will convince the reader that for any collection of $0\leq k_2,\ldots, k_{l_0}\leq n-1$ such that $k\geq n$, the double sum in \eqref{sum10} will be $\geq$ the supremum over all corresponding sums with $0\leq k_2',\ldots, k'_{l'_0}\leq n-1$ such that $k':=k'_2+\ldots+k'_{l'_0}=n$. Now, with $k=n$ we obtain in \eqref{sum10}
\begin{equation}\Label{sum20}
\begin{aligned}
d\geq\sum_{l=2}^{l_0}\sum_{j=0}^{k_l-1}(n-j) &= \sum_{l=2}^{l_0}\left(k_ln-\frac{k_l(k_l+1)}{2}\right)\\
&=\left(n+\frac{1}{2}\right)k-\frac{1}{2}\sum_{l=2}^{l_0}k_l^2\\
&=\left(n+\frac{1}{2}\right)n-\frac{1}{2}\sum_{l=2}^{l_0}k_l^2
&\geq \frac{n(n+1)}{2},
\end{aligned}
\end{equation}
where in the last step we used that fact that
$$
\sum_{l=2}^{l_0}k_l^2\leq \left(\sum_{l=2}^{l_0}k_l\right)^2=k^2=n^2.
$$
Since $d\leq N-n<n(n+1)/2$, we conclude that also the second condition in \eqref{conds0} must hold. Of course, the conclusion of Theorem \ref{Main0} will have no meaning unless $n+d+k< N$.
}
\end{Rem}

\section{Preliminaries}\Label{prelim} We shall use the setup, notation and conventions in \cite{EHZ04}. For the reader's convenience, we recall the setup here. Let $M$ be a strictly pseudoconvex hypersurface in $\bC^{n+1}$. In particular, $M=M^n$ is CR manifold of real dimension $2n+1$ with a rank $n$ CR bundle (the bundle of (0,1)-vector fields on $M$), denoted $T^{0,1}M$.  Near a point $p\in M$, we let $\theta$ be a contact form (fixing a pseudohermitian structure as in \cite{Webster78}) and $T$ its characteristic (or Reeb) vector field, i.e.\ T is the unique real vector field satisfying $T \lrcorner d\theta = 0$ and $\langle \theta, T \rangle = 1$.  We complete $\theta$ to an admissible coframe $(\theta, \theta^{1},\ldots,\theta^{n})$ for the bundle $T'M$ of $(1,0)$-cotangent vectors (the cotangent vectors that annihilate $T^{0,1}M)$.  The coframe is called admissible (following the terminology in \cite{Webster78}) if $\langle \theta^{\alpha}, T \rangle = 0$, for $\alpha = 1, \ldots, n$.
We let $L_{1},\ldots,L_{n}$ be a frame for the bundle $T^{0,1}M$, near $p$, such that $(T, L_{1},\ldots,L_{n},L_{\bar{1}},\ldots,L_{\bar{n}})$ is a frame for $\mathbb{C}TM$, dual to the coframe $(\theta, \theta^{1},\ldots,\theta^{n},\theta^{\bar{1}},\ldots,\theta^{\bar{n}})$.  We use the notation that $L_{\bar{\alpha}} = \bar{L}_{\alpha}$, etc.  Relative to this frame, let $(g_{\alpha\bar{\beta}})$ denote the matrix of the Levi form.  We shall, as we may, require that our choice of admissible coframe is such that the Levi form is represented by the identity matrix, $g_{\alpha\bar\beta}=\delta_{\alpha\bar\beta}$, although we shall retain the notation $g_{\alpha\bar\beta}$ (and not always use the fact that it is the identity matrix).

We denote by $\nabla$ the Tanaka-Webster pseudohermitian connection (\cite{Tanaka75}, \cite{Webster78}), given relative to the chosen frame and coframe by
\begin{equation*}
\nabla L_{\alpha} := \omega_{\alpha}^{\:\:\beta}\otimes L_{\beta},
\end{equation*}
where the connection 1-forms $\omega_{\alpha}^{\:\:\beta}$ are determined by the conditions
\begin{equation}\Label{psh1}
\begin{aligned}
d\theta^{\beta} &= \theta^{\alpha}\wedge\omega_{\alpha}^{\:\:\beta} \qquad \textrm{mod}\theta\wedge\theta^{\bar{\alpha}}, \\
dg_{\alpha\bar{\beta}} &= \omega_{\alpha\bar{\beta}} + \omega_{\bar{\beta}\alpha}.
\end{aligned}
\end{equation}
Here and for the remainder of this paper, we use the summation convention (indices appearing both as subscripts and superscripts are summed over) and the Levi form to lower and raise indices as usual, e.g.\ $\omega_{\alpha\bar\beta}:=g_{\gamma\bar\beta}\omega_{\alpha}{}^\gamma$; moreover, lower case Greek indices $\alpha,\beta$, etc.\ will run over the index set $\{1,2,\ldots, n\}$. We may rewrite the first equation in \eqref{psh1} as
\begin{equation}\Label{psh2}
d\theta^{\beta} = \theta^{\alpha}\wedge\omega_{\alpha}^{\:\:\:\beta} + \theta\wedge\tau^{\beta}, \qquad \tau^{\beta} = A^{\beta}_{\:\:\:\bar{\nu}}\theta^{\bar{\nu}}, \qquad A^{\alpha\beta} = A^{\beta\alpha}
\end{equation}
for a suitably determined torsion matrix $(A^{\beta}_{\:\:\:\bar{\nu}})$.  We also recall the fact that the coframe $(\theta, \theta^{1}, \ldots, \theta^{n})$ is admissible if and only if $d\theta = i g_{\alpha \bar{\beta}}\theta^{\alpha}\wedge \theta^{\bar{\beta}}$. For a fixed pseudohermitian structure $\theta$ and choice of Levi form $g_{\alpha\bar\beta}$, the 1-forms $\theta^\alpha$ in an admissible coframe are determined up to unitary transformations.

Let $f:M \rightarrow \hat{M} $ be a transversal, local CR embedding of a strictly pseudoconvex, hypersurface $M\subset \mathbb{C}^{n+1}$ into another such $\hat{M}\subset \mathbb{C}^{N+1}$. Due to the strict pseudoconvexity of $M$ and $\hat M$, as mentioned in the introduction, $f$ is transversal and a local embedding if and only if $f$ is non-constant. By Corollary 4.2 in \cite{EHZ04}, given any admissible coframe $(\theta, \theta^\alpha )=(\theta, \theta^\alpha )_{\alpha=1}^{n}$ as above, defined locally near $p\in M$ there exist admissible coframes $( \hat{\theta}, \hat{\theta}^A )=( \hat{\theta}, \hat{\theta}^A )_{A=1}^{N}$ on $\hat{M}$, defined near $f(p) \in \hat{M}$, so that:
\begin{equation}
f^{\ast}(\hat{\theta},\: \hat{\theta}^\alpha, \: \hat{\theta}^a)=(\theta,\: \theta^\alpha,\: 0),
\label{adaptedcond}
\end{equation}
and such that also $\hat g_{A\bar B}=\delta_{A\bar B}$. Here and for the remainder of this paper, we shall use the convention that lower case Greek indices, as mentioned above, vary from 1 to n, capital Roman letters vary from 1 to $N$, i.e.\ $A,B\ldots \in \{1,...,N \}$, and lower case Roman letters vary in the normal direction, i.e.\ $a,b,\ldots \in \{n+1,...,N \} $. We shall say that the coframe $( \hat{\theta}, \hat{\theta}^A )$ is adapted to $f$ with respect to $(\theta, \theta^\alpha )$ if it satisfies \eqref{adaptedcond} and the Levi form in this coframe is the identity matrix.

Equation \eqref{psh2} implies that when $(\theta, \theta^{A})$ is adapted to $M$, if the pseudoconformal connection matrix of $(\hat{M}, \hat{\theta})$ is $\hat{\omega}_{B}^{\:\:\:A}$, then that of $(M, \theta)$ is the pullback of $\hat{\omega}_{\beta}^{\:\:\:\alpha}$.  The pulled back torsion $\hat{\tau}^{\alpha}$ is $\tau^{\alpha}$, so omitting the $\hat{}$ over these pullbacks will not cause any ambiguity and we shall do so from now on.

The matrix of 1-forms $(\omega_{\alpha}^{\:\:\:b})$ pulled back to $M$ defines the \emph{second fundamental form} of the embedding $f:M \rightarrow \hat{M}$, denoted $\Pi_f\colon T^{1,0}M\times T^{1,0}M\to T^{1,0}\hat M/f_*T^{1,0}M$, as follows.  Since $\theta^{b} = 0$ on $M$, equation \eqref{psh2} implies that on $M$,
\begin{equation}
\omega_{\alpha}^{\:\:\:b} \wedge\theta^{\alpha} + \tau^{b}\wedge\theta = 0,
\end{equation}
which implies that
\begin{equation}\Label{EHZ4.4}
\omega_{\alpha}^{\:\:\:b} = \omega_{\alpha \:\:\: \beta}^{\:\:\:b}\theta^{\beta}, \qquad \omega_{\alpha \:\:\: \beta}^{\:\:\:b} = \omega_{\beta \:\:\: \alpha}^{\:\:\:b}, \qquad \tau^{b} = 0.
\end{equation}
The second fundamental form is now defined by
\begin{equation}\Label{SFFdef}
\Pi_f(L_\alpha,L_\beta):=\omega_\alpha{}^a{}_\beta [L_a],
\end{equation}
where $[L_a]$ denotes the equivalence class of the normal vector fields $L_a$ in $T^{1,0}\hat M/f_*T^{1,0}M$. Following \cite{EHZ04} we identify the CR-normal space $T_{f(p)}^{1,0}\hat{M}/f_*T_{p}^{1,0}M$, also denoted by $N_{p}^{1,0}{M}$, with $\mathbb{C}^{N-n}$ by choosing the equivalence classes of $L_{a}$ as a basis. Therefore, for fixed $\alpha, \beta$, we view the component vector $(\omega_{\alpha}{}^a{}_ \beta)_{a=n+1}^{N}$ as an element of $\mathbb{C}^{N-n}$.  By also viewing the second fundamental form as a section over $M$ of the bundle
$T^{1,0}M\otimes N^{1,0}{M} \otimes T^{1,0}M$, we may use the pseudohermitian connections on $M$ and $\hat{M}$ to define the covariant differential
\begin{equation*}
\nabla \omega_{\alpha\:\:\beta}^{\:\:a} = d\omega_{\alpha\:\:\beta}^{\:\:a} - \omega_{\mu\:\:\beta}^{\:\:a}\omega_{\alpha}^{\:\:\mu} + \omega_{\alpha\:\:\beta}^{\:\:b}\omega_{b}^{\:\:a} - \omega_{\alpha\:\:\mu}^{\:\:a}\omega_{\beta}^{\:\:\mu}.
\end{equation*}
We write $\omega_{\alpha\:\:\beta ; \gamma}^{\:\:a}$ to denote the component in the direction $\theta^{\gamma}$ and define higher order derivatives inductively as:
\begin{equation*}
\nabla \omega_{\gamma_{1}\:\:\gamma_{2};\gamma_{3}\ldots\gamma_{j}}^{\:\:a} = d\omega_{\gamma_{1}\:\:\gamma_{2};\gamma_{3}\ldots\gamma_{j}}^{\:\:a} + \omega_{\gamma_{1}\:\:\gamma_{2};\gamma_{3}\ldots\gamma_{j}}^{\:\:b}\omega_{b}^{\:\:a} - \sum_{l=1}^{j}\omega_{\gamma_{1}\:\:\gamma_{2};\gamma_{3}\ldots\gamma_{l-1}\mu
\gamma_{l+1}\ldots\gamma_{j}}^{\:\:a}\omega_{\gamma_{l}}^{\:\:\mu}.
\end{equation*}
We also consider the component vectors of higher order derivatives as elements of $\mathbb{C}^{N-n}\cong N_p^{1,0}M$ and define an increasing sequence of vector spaces
\begin{equation*}
E_{2}(p) \subseteq \ldots \subseteq E_{l}(p) \subseteq \ldots \subseteq \mathbb{C}^{N-n}\cong N_p^{1,0}M
\end{equation*}
by letting $E_{l}(p)$ be the span of the vectors
\begin{equation}\Label{Eldef}
(\omega_{\gamma_{1}\:\:\gamma_{2};\gamma_{3}\ldots\gamma_{j}}^{\:\:a})_{a=n+1}^{N}, \qquad \forall\, 2 \leq j \leq l, \gamma_{j}\in \{1,\ldots,n\},
\end{equation}
evaluated at $p \in M$.  The Levi form defines an isomorphism $$T'_{f(p)}\hat M/\hat E_1(p)\cong T^{1,0}_{f(p)}\hat M/f_*T^{1,0}_pM=N^{1,0}_pM,$$
and it is shown in \cite{EHZ04} (Sections 4 and 7) that $E_l(p)\cong \hat E_l(p)/\hat E_1(p)$,  where the $\hat E_l(p)$ are as defined in \eqref{hatEkdef}. We let $d_l(p)$ be the dimension of $E_l(p)$, which is then consistent the definition \eqref{dkdef}. 

We shall also need the pseudoconformal connection and structure equations introduced by Chern and Moser in \cite{CM74}.  Let $Y$ be the bundle of coframes $(\omega,\omega^{\alpha},\omega^{\bar{\alpha}}, \phi)$ on the real ray bundle $\pi_{E}:E\rightarrow M$ of all contact forms defining the same orientation of $M$, such that $d\omega = ig_{\alpha\bar{\beta}}\omega^{\alpha}\wedge\omega^{\bar{\beta}} + \omega\wedge\phi$ where $\omega^{\alpha} \in \pi_{E}^{*}(T'M)$ and $\omega$ is the canonical 1-form on $E$.  In \cite{CM74} it was shown that these forms can be completed to a full set of invariants on $Y$ given by the coframe of 1-forms
\[ (\omega, \omega^{\alpha}, , \omega^{\bar{\beta}}, \phi, \phi_{\beta}{}^{\alpha}, \phi_{\bar\beta}{}^{\bar\alpha}, \phi^{\alpha}, \phi^{\bar{\alpha}}, \psi) \] which define the pseudoconformal connection on $Y$.  These forms satisfy the following structure equations, which we will use extensively (see \cite{CM74} and its appendix):
\begin{equation}\Label{CMstructure}
\begin{aligned}
& \phi_{\alpha\bar{\beta}} + \phi_{\bar{\beta}\alpha} = g_{\alpha\bar{\beta}}\phi, \nonumber \\
& d\omega = i\omega^{\mu}\wedge\omega_{\mu} + \omega \wedge \phi, \nonumber \\
& d\omega^{\alpha} = \omega^{\mu}\wedge\phi_{\mu}^{\:\:\alpha} + \omega\wedge\phi^{\alpha}, \nonumber \\
& d\phi = i\omega_{\bar{\nu}}\wedge\phi^{\bar{\nu}} + i\phi_{\bar{\nu}}\wedge\omega^{\bar{\nu}} + \omega \wedge \psi, \nonumber \\
& d\phi_{\beta}^{\:\:\alpha} = \phi_{\beta}^{\:\:\mu}\wedge\phi_{\mu}^{\:\:\alpha} + i\omega_{\beta}\wedge\phi^{\alpha} - i\phi_{\beta}\wedge\omega^{\alpha} -i\delta_{\beta}^{\:\:\alpha}\phi_{\mu}\wedge\omega^{\mu} - \frac{\delta_{\beta}^{\:\:\alpha}}{2}\psi\wedge\omega + \Phi_{\beta}^{\:\:\alpha}, \nonumber \\
& d\phi^{\alpha} = \phi\wedge\phi^{\alpha} + \phi^{\mu}\wedge\phi_{\mu}^{\:\:\alpha} - \frac{1}{2}\psi\wedge\omega^{\alpha} + \Phi^{\alpha}, \nonumber \\
& d\psi = \phi\wedge\psi + 2i\phi^{\mu}\wedge\phi_{\mu} + \Psi.
\end{aligned}
\end{equation}
Here the 2-forms $\Phi_{\beta}^{\:\:\alpha}, \Phi^{\alpha}, \Psi$ constitute the \emph{pseudoconformal curvature} of $M$.  We may decompose $\Phi_{\beta}^{\:\:\alpha}$ as follows
\begin{equation*}
\Phi_{\beta}^{\:\:\alpha} = S_{\beta\:\:\:\mu\bar{\nu}}^{\:\:\alpha}\omega^{\mu}\wedge\omega^{\bar{\nu}} + V_{\beta\:\:\:\mu}^{\:\:\alpha}\omega^{\mu}\wedge\omega + V^{\alpha}_{\:\:\:\beta\bar{\nu}}\omega\wedge\omega^{\bar{\nu}}.
\end{equation*}
We will also refer to the tensor $S_{\beta\:\:\:\mu\bar{\nu}}^{\:\:\alpha}$ as the pseudoconformal curvature of $M$ (as $S_{\beta\:\:\:\mu\bar{\nu}}^{\:\:\alpha}$ and its covariant derivatives in fact determine $\Phi_{\beta}^{\:\:\alpha}, \Phi^{\alpha}, \Psi$).  The curvature tensor $ S_{\beta\:\:\:\mu\bar{\nu}}^{\:\:\alpha}$ is required to satisfy certain trace and symmetry conditions (see \cite{CM74}), but for the purposes of this paper, the important point to emphasize is that \emph{for a sphere, the pseudoconformal curvature vanishes}.

If we fix a contact form $\theta$, i.e.\ a section $M \rightarrow E$, then any admissible coframe $(\theta, \theta^{\alpha})$ for $M$ defines a unique section $M \rightarrow Y$ under which the pullbacks of $(\omega, \omega^{\alpha})$ coincide with $(\theta, \theta^{\alpha})$ and the pullback of $\phi$ vanishes.  As in \cite{Webster78} we use this section to pull the pseudoconformal connection forms back to $M$.  We can express the pulled back tangential pseudoconformal curvature tensor $ S_{\beta\:\:\:\mu\bar{\nu}}^{\:\:\alpha}$ in terms of the tangential pseudohermitian curvature tensor $ R_{\beta\:\:\:\mu\bar{\nu}}^{\:\:\alpha}$ by
\begin{equation}\Label{CRpshcurv}
 S_{\alpha\bar{\beta}\mu\bar{\nu}} = R_{\alpha\bar{\beta}\mu\bar{\nu}} - \frac{R_{\alpha\bar{\beta}}g_{\mu\bar{\nu}} + R_{\mu\bar{\beta}}g_{\alpha\bar{\nu}} + R_{\alpha\bar{\nu}}g_{\mu\bar{\beta}} + R_{\mu\bar{\nu}}g_{\alpha\bar{\beta}}}{n+2} + \frac{R(g_{\alpha\bar{\beta}}g_{\mu\bar{\nu}} + g_{\alpha\bar{\nu}}g_{\mu\bar{\beta}})}{(n+1)(n+2)},
\end{equation}
where
\begin{equation*}
R_{\alpha\bar{\beta}} := R_{\mu\:\:\:\alpha\bar{\beta}}^{\:\:\mu} \qquad \textrm{and } R:= R_{\mu}^{\:\:\mu}
\end{equation*}
are respectively the pseudohermitian Ricci and scalar curvature of $(M, \theta)$.  This formula expresses the fact that $S_{\alpha\bar{\beta}\mu\bar{\nu}}$ is the ``traceless component" of $R_{\alpha\bar{\beta}\mu\bar{\nu}}$ with respect to the decomposition of the space of all tensors with the symmetry conditions of $S_{\alpha\bar{\beta}\mu\bar{\nu}}$ into the direct sum of the subspace of tensors with trace zero and the subspace of \emph{conformally flat tensors}, i.e. tensors of the form
\begin{equation}\Label{confflat}
T_{\alpha\bar{\beta}\mu\bar{\nu}} = H_{\alpha\bar{\beta}}g_{\mu\bar{\nu}} + H_{\mu\bar{\beta}}g_{\alpha\bar{\nu}} + H_{\alpha\bar{\nu}}g_{\mu\bar{\beta}} + H_{\mu\bar{\nu}}g_{\alpha\bar{\beta}},
\end{equation}
where $(H_{\alpha\bar{\beta}})$ is any Hermitian matrix.  Observe that covariant derivatives of conformally flat tensors are conformally flat, since $\nabla g_{\alpha\bar{\beta}} = 0$ by definition (see the second equation of \eqref{psh1}).

The following result relates the pseudoconformal and pseudohermitian connection forms.  It is alluded to in \cite{Webster78} and a proof may be found in \cite{EHZ04}, where the result appears as Proposition 3.1. 

\begin{Pro}\Label{EHZProp3.1}
Let $M$ be a smooth Levi-nondegenerate CR-manifold of hypersurface type with CR dimension $n$, and with respect to an admissible coframe $(\theta, \theta^{\alpha})$ let the pseudoconformal and pseudohermitian connection forms be pulled back to $M$ as above.  Then we have the following relations:
\begin{equation}
\phi_{\beta}^{\:\:\alpha} = \omega_{\beta}^{\:\:\alpha} + D_{\beta}^{\:\:\alpha}\theta, \qquad \phi^{\alpha} = \tau^{\alpha} + D_{\mu}^{\:\:\alpha}\theta^{\mu} + E^{\alpha}\theta, \qquad \psi = iE_{\mu}\theta^{\mu} - iE_{\bar{\nu}}\theta^{\bar{\nu}} + B\theta,
\end{equation}
where
\begin{equation}
\begin{aligned}
&D_{\alpha\bar{\beta}} := \frac{iR_{\alpha\bar{\beta}}}{n+2} - \frac{iRg_{\alpha\bar{\beta}}}{2(n+1)(n+2)}, \nonumber \\
&E^{\alpha} := \frac{2i}{2n + 1}(A^{\alpha\mu}_{\:\:\:\:;\mu} - D^{\bar{\nu}\alpha}_{\:\:\:\:;\bar{\nu}}), \nonumber \\
&B := \frac{1}{n}(E^{\mu}_{\:\:\:;\mu} + E^{\bar{\nu}}_{\:\:\:;\bar{\nu}} - 2A^{\beta\mu}A_{\beta\mu} + 2D^{\bar{\nu}\alpha}D_{\bar{\nu}\alpha}).
\end{aligned}
\end{equation}
\end{Pro}

If $(\theta,\theta^A)$ is an admissible coframe on $\hat M$ near $\hat p:=f(p)$ relative to the local embedding $f\colon M\to\hat M$ and the coframe $(\theta,\theta^\alpha)$, then we can also pull back the corresponding Chern-Moser connection forms
\[ (\hat \omega, \hat \omega^{A}, , \hat \omega^{\bar{B}}, \hat\phi, \hat \phi_{B}{}^{A}, \hat \phi_{\bar B}{}^{\bar A}, \hat \phi^{A}, \hat \phi^{\bar{A}}, \hat \psi) \]
to $\hat M$ and then pull these forms back to $M$ using $f$. The second fundamental form of $f$ is controlled by the CR Gauss Equation (see equation (5.8) in \cite{EHZ04}), which in the case where $\hat M$ is the sphere $\bS^N$ (so that the CR curvature of the target vanishes) takes the form
\begin{equation}
\begin{aligned}
S_{\alpha \bar{\beta} \mu \bar{\nu}} = & - g_{a\bar{b}}\omega_{\alpha \: \: \: \mu}^{\:\:\:a}\omega_{\bar{\beta} \: \: \: \bar{\nu}}^{\:\:\:\bar{b}}+\frac{1}{n+2}(\omega_{\gamma \: \: \: \alpha}^{\:\:\:a}\omega_{\: \: \: a \bar{\beta}}^{\gamma}g_{\mu \bar{\nu}}+\omega_{\gamma \: \: \: \mu}^{\:\:\:a}\omega_{\: \: \: a \bar{\beta}}^{\gamma}g_{\alpha \bar{\nu}}   + \omega_{\gamma \: \: \: \alpha}^{\:\:\:a}\omega_{\: \: \: a \bar{\nu}}^{\gamma}g_{\mu \bar{\beta}}\\ &+ \omega_{\gamma \: \: \: \mu}^{\:\:\:a}\omega_{\: \: \: a \bar{\nu}}^{\gamma}g_{\alpha \bar{\beta}} )- \frac{\omega_{\gamma \: \: \: \delta}^{\:\:\:a}\omega_{\:\:\:a}^{\gamma \: \: \: \delta}}{(n+1)(n+2)}(g_{\alpha \bar{\beta}}g_{\mu \bar{\nu}}+g_{\alpha \bar{\nu}}g_{\mu \bar{\beta}}).
\end{aligned}
\end{equation}
The main focus in this paper is to describe the geometry of the mapping $f$ in terms of a {\it given} second fundamental form, and therefore the Gauss Equation will not play a role in what follows.

\section{Partial Rigidity for degenerate mappings into spheres} \Label{PRdegsect}

We shall now enter the proof of Theorem \ref{Main0}. We will use the setup and notation introduced in the previous section. We begin by establishing some identities that will be used in the proof.

\subsection{Some identities of "Codazzi type"} It follows from (3.8) in \cite{EHZ04} that along $M$ we have
\begin{multline}\Label{e:100}
d{\o_\a}^a-{\o_\a}^\gamma\wedge{\omega_\gamma}^a-{\o_\a}^b\wedge{\omega_b}^a=\\
{\hat R}_{\a}{}^a{}_{\mu\bar\nu}
\theta^\mu\wedge\theta^{\bar\nu}
+\hat W_\a{}^a{}_\mu\theta^\mu\wedge\theta
- \hat W^a{}_{\a\bar\nu}\theta^{\bar \nu}\wedge \theta,
\end{multline}
where we have used the fact (see \eqref{EHZ4.4}) that $\tau^a=\theta^a=0$ on $M$. Since ${\omega_\alpha}^a={{\omega_\alpha}^a}_\beta\theta^\beta$, we also obtain (via \eqref{psh2}),
$$
d{\o_\a}^a=d{{\omega_\alpha}^a}_\beta\wedge \theta^\beta+{{\omega_\alpha}^a}_\beta(\theta^\mu\wedge{\omega_\mu}^\beta+
\theta\wedge\tau^\beta)
$$
and, hence,
$$
d{\o_\a}^a-{\o_\a}^\gamma\wedge{\omega_\gamma}^a-{\o_\a}^b\wedge{\omega_b}^a=
(d{{\omega_\alpha}^a}_\beta-\omega_\gamma{}^a{}_{\beta}{\o_\alpha}^\gamma-
{{\omega_\a}^a}_\gamma{\o_\beta}^\gamma + {{\omega_\alpha}^b}_\beta{\o_b}^a)\wedge \theta^\beta-{{\omega_\alpha}^a}_\beta\tau^\beta\wedge\theta.
$$
We conclude that \eqref{e:100} can be rewritten as
\begin{equation}\Label{e:200}
(\nabla \o_\a{}^a{}_\beta+{\hat R}_{\a}{}^a{}_{\beta\bar\nu}
\theta^{\bar\nu})\wedge \theta^\beta+\\
(\hat W_\a{}^a{}_\mu\theta^\mu
- \hat W^a{}_{\a\bar\nu}\theta^{\bar \nu}-{{\omega_\alpha}^a}_\beta\tau^\beta)\wedge \theta=0.
\end{equation}
It follows that the covariant derivatives $\omega_\a{}^a{}_{\b;\gamma}$ are symmetric in $\alpha,\beta,\gamma$ and the following identities hold:
\begin{equation}\Label{e:300}
\begin{aligned}
\o_\a{}^a{}_{\beta;\bar\nu}+{\hat R}_{\a}{}^a{}_{\beta\bar\nu} &=0\\
\o_\a{}^a{}_{\beta;0}-\hat W_\a{}^a{}_\beta &=0\\
\o_\a{}^a{}_{\beta}A^\beta{}_{\bar\nu}+\hat W^a{}_{\a\bar\nu} &=0.
\end{aligned}
\end{equation}
Moreover, since the target is the sphere with $\hat S_{A\bar B C\bar D}=0$, we also conclude, by \eqref{CRpshcurv}, that
\begin{equation}
\hat R_\a{}^a{}_{\b \bar\nu}=\frac{1}{N+2}(\hat R_\a{}^ag_{\b\bar\nu}+\hat R_\b{}^ag_{\a\bar\nu}),
\end{equation}
where $\hat R_{A\bar B}$ denotes the Ricci curvature tensor.
By Proposition \ref{EHZProp3.1}, we obtain
\begin{equation}
\hat R_\a{}^a{}_{\b\bar\nu}=-i(\hat D_\a{}^ag_{\b\bar\nu}+\hat D_\b{}^a g_{\a\bar\nu}),
\end{equation}
and, hence, the first equation in \eqref{e:300} yields
\begin{equation}\Label{barder}
\o_\a{}^a{}_{\beta;\bar\nu}=i(\hat D_\a{}^ag_{\b\bar\nu}+\hat D_\b{}^a g_{\a\bar\nu}).
\end{equation}

\subsection{Proof of Theorem $\ref{Main0}$} We first note that to prove the theorem it suffices to show that $f(U)$, where $U\subset M$ is as described in the theorem, is contained in a complex plane of dimension $n+d+k+1$ by unique continuation along the connected, minimal hypersurface $M$. (The CR mapping $f$ extends as a holomorphic mapping in a connected open set with $M$ in its boundary.) In what follows, we shall assume that $d\geq 1$.  (As is well known, $d=0$ can only happen when $M$ is locally spherical and $f$ is totally geodesic, i.e.\ contained in a plane of dimension $n+1$.) By making an initial unitary transformation of the normal vector fields $L_a$ we may assume, without loss of generality, that
\begin{equation}\Label{k0deg}
\span \{\o_{\gamma_1}{}^\#{}_{\gamma_2;\gamma_3,\ldots,\gamma_{l}} L_\#,
\; 2\le l\le l_0 \}= \span \{L_\#\},
\quad
\omega_{\gamma_1}{}^j{}_{\gamma_2;\gamma_3\ldots\gamma_l}\equiv 0,
\quad l\ge 2,
\end{equation}
where $\#,*$ etc run over the indices $n+1,\ldots, n+d$, and $i,j$ over the remaining indices in the codimensional range $n+d+1,\ldots, N$ (unless otherwise specified). Recall that each $E_l(p)$, $2\leq l\leq l_0$, has locally constant dimension $d_l$ near $p_0$ and
$$
0<d_2<d_3<\ldots<d_{l_0}=d.
$$
By applying Gram-Schmidt to the $L_\#$, we may further assume that, for each $2\leq l\leq l_0$,
\begin{equation}\Label{strongnorm}
\omega_{\gamma_1}{}^\#{}_{\gamma_2;\gamma_3\ldots\gamma_l}\equiv 0, \quad \#\geq n+d_l+1.
\end{equation}

We take as our starting point the following identities (see \cite{EHZ04}, (9.4-9.6))
\begin{equation}\Label{e:400}
\omega_\#{}^j{}_\mu=0,
\end{equation}
\begin{equation}\Label{eq-phiia}
\hat\phi_\a{}^j=\hat D_\a{}^j\theta,
\quad \hat\phi^j=\hat D_\mu{}^j\theta^\mu+\hat E^j\theta,
\end{equation} and
\begin{equation}\Label{eq-phiia2}
\hat\phi_\a{}^\#=\omega_\a{}^\#{}_\mu\theta^\mu
+\hat D_\a{}^\#\theta,
\quad \hat\phi^\#=\hat D_\mu{}^\#\theta^\mu
+\hat E^\#\theta,
\end{equation}
which follow exactly as in the beginning of the proof of Theorem 2.2 in Section 9 of \cite{EHZ04}.
Differentiating $\hat\phi_\a{}^j$  we obtain
\begin{equation}\Label{e:450}
d\hat\phi_\a{}^j= d\hat D_\a{}^j\wedge\theta+ig_{\mu\bar\nu}\hat D_{\a}{}^j\theta^\mu\wedge\theta^{\bar\nu}
\end{equation}
and by the structure equations on $\bS=\bS^N$,
\begin{equation}\Label{e:500}
\begin{aligned}
d\hat\phi_\a{}^j=&
\hat\phi_\a{}^\mu\wedge \hat\phi_\mu{}^j+\hat\phi_\a{}^a\wedge\hat\phi_a{}^j+i\theta_\a\wedge\hat\phi^j\\=&
\hat D_\mu{}^j \omega_\a{}^\mu{}\wedge\theta
+\omega_\a{}^\#{}_\beta\omega_\#{}^j{}_{\bar\nu}\theta^\beta\wedge\theta^{\bar\nu}+
\omega_\a{}^\#{}_\beta(\hat D_\#{}^j+\omega_\#{}^j{}_{0}) \theta^\b\wedge \theta\\&
-\hat D_\a{}^a\omega_a{}^j\wedge\theta
-ig_{\a\bar\nu}\hat D_\beta{}^j\theta^\beta\wedge \theta^{\bar\nu}+ig_{\a\bar\nu}\hat E^j\theta^{\bar\nu}\wedge\theta.
\end{aligned}
\end{equation}
We have used here, and will do so throughout this paper, the fact that the CR curvature terms on $\bS$ vanish. By identifying terms and using the definition of covariant derivatives, we obtain the following three identities from \eqref{e:450} and \eqref{e:500}
\begin{equation}\Label{e:600}
\begin{aligned}
\omega_\a{}^\#{}_\beta\omega_\#{}^j{}_{\bar\nu}& =i(g_{\b\bar\nu}\hat D_\a{}^j+g_{\a\bar\nu}\hat D_\b{}^j)\\
\hat D_\a{}^j{}_{;\,\bar\nu}& =ig_{\a\bar\nu}\hat E^j\\
\hat D_\a{}^j{}_{;\,\beta}& =\omega_\a{}^\#{}_\beta(\hat D_\#{}^j+\omega_\#{}^j{}_{0})
\end{aligned}
\end{equation}
Next, we note from Proposition \ref{EHZProp3.1} and \eqref{e:400} that
\begin{equation}\Label{e:610}
\hat\phi_\#{}^{j}=\o_\#{}^j{}_{\bar\nu}\theta^{\bar\nu}+(\hat D_\#{}^{j}+\o_\#{}^{j}{}_0)\theta,
\end{equation}
which implies (via \eqref{psh2}) that
\begin{equation}\Label{e:620}
d\hat\phi_\#{}^j=d\o_\#{}^j{}_{\bar\nu}\wedge \theta^{\bar\nu}-\o_\#{}^j{}_{\bar\gamma}\o_{\bar\nu}{}^{\bar\gamma}\wedge \theta^{\bar\nu}+ig_{\mu\bar\nu}(\hat D_\#{}^{j}+\o_\#{}^{j}{}_0)\theta^\mu\wedge\theta^{\bar\nu}\mod\theta
\end{equation}
The structure equation for $\hat\phi_\#{}^{j}$ reduces to (using the fact that $\theta_\#=0$ and $\phi_\mu{}^{j'}=0$ on $M$),
\begin{equation}
\begin{aligned}
d\hat\phi_\#{}^{j}=&
\hat\phi_\#{}^a\wedge\hat\phi_a{}^{j}\\ =&\hat\phi_\#{}^{*}\wedge\hat\phi_{*}{}^{j} + \hat\phi_{\#}{}^{i}\wedge\hat\phi_{i}{}^{j}\\ =&\o_\#{}^{*}\wedge\o_{*}{}^{j} + \o_{\#}{}^{i}\wedge\o_{i}{}^{j}\mod\theta,
\end{aligned}
\end{equation}
where the last identity follows from  the identities in Proposition \ref{EHZProp3.1}. By using \eqref{e:400} again, we notice that
\begin{equation}\Label{e:630}
d\hat\phi_\#{}^{j}=\o_{*}{}^{j}{}_{\bar\nu}\o_\#{}^{*}\wedge\theta^{\bar \nu} - \o_{\#}{}^{i}{}_{\bar\nu}\o_{i}{}^{j}\wedge\theta^{\bar\nu}\mod\theta.
\end{equation}
Now, it follows from \eqref{e:400} again that
$$
\o_{*}{}^{j}{}_{\bar\nu}\o_\#{}^{*}{}_\mu=\o_{a}{}^{j}{}_{\bar\nu}\o_\#{}^{a}{}_\mu,\quad
\o_{\#}{}^{i}{}_{\bar\nu}\o_{i}{}^{j}{}_\mu=\o_{\#}{}^{a}{}_{\bar\nu}\o_{a}{}^{j}{}_\mu,
$$
and, hence, by identifying the coefficient in front of $\theta^\mu\wedge\theta^{\bar\nu}$ in the two identities \eqref{e:620} and \eqref{e:630} we deduce that
\begin{equation}\Label{e:640}
\o_\#{}^j{}_{\bar\nu;\mu}+ig_{\mu\bar\nu}(\hat D_\#{}^{j}+\o_\#{}^{j}{}_0)=0.
\end{equation}
We now observe that covariant differentiation of $\o_{\gamma_1}{}^\#{}_{\gamma_2;\gamma_3\ldots\gamma_l}\o_\#{}^j{}_{\bar\nu}$ with respect to $\theta^\mu$ will only involve the components of the tensor $(\o_{\a_1}{}^\#{}_{\a_2;\a_3\ldots\a_l}\o_\#{}^a{}_{\bar\nu})$ with $a\in \{n+d+1,\ldots,N\}$ in view of the identity \eqref{e:400}. Thus, by covariantly differentiating the first identity in \eqref{e:600} with respect to $\theta^{\mu}$, we obtain
\begin{equation}\Label{e:635}
\begin{aligned}
\omega_\a{}^\#{}_{\beta;\mu}\omega_\#{}^j{}_{\bar\nu}& = -\omega_\a{}^\#{}_{\beta}\omega_\#{}^j{}_{\bar\nu;\mu}+i(g_{\b\bar\nu}\hat D_\a{}^j{}_{;\mu}+g_{\a\bar\nu}\hat D_\b{}^j{}_{;\mu})\\
&=i(g_{\b\bar\nu}\hat D_\a{}^j{}_{;\mu}+g_{\a\bar\nu}\hat D_\b{}^j{}_{;\mu})+
ig_{\mu\bar\nu}\omega_\a{}^\#{}_{\beta}(\hat D_\#{}^{j}+\o_\#{}^{j}{}_0)\\
& =i (g_{\a\bar\nu}\omega_\b{}^\#{}_{\mu}+g_{\b\bar\nu}\omega_\mu{}^\#{}_{\a}
+g_{\mu\bar\nu}\omega_\a{}^\#{}_{\beta})
(\hat D_\#{}^{j}+\o_\#{}^{j}{}_0),
\end{aligned}
\end{equation}
where the last identity follows from the third identity in \eqref{e:600}. For reference, we also note that this can be written as
\begin{equation}\Label{omega3ref}
\omega_\a{}^\#{}_{\beta;\mu}\omega_\#{}^j{}_{\bar\nu} = i (g_{\a\bar\nu}\hat D_\b{}^j{}_{;\mu}+g_{\b\bar\nu}\hat D_\mu{}^j{}_{;\a}
+g_{\mu\bar\nu}\hat D_\a{}^j{}_{;\b}),
\end{equation}
A simple induction (using also the fact that covariant derivativates of $\o_\a{}^a{}_\b$ in the $\theta^\gamma$ directions are symmetric in their indices; cf.\ \eqref{e:200}) shows that, for all $3\leq l$, we have
\begin{equation}\Label{e:636}
\omega_{\gamma_1}{}^\#{}_{\gamma_2;\gamma_3\ldots\gamma_l}\omega_\#{}^j{}_{\bar\nu}=
i\{g_{\gamma_1\bar\nu}\omega_{\gamma_2}{}^\#{}_{\gamma_3;\gamma_4\ldots\gamma_l}\}
(\hat D_\#{}^{j}+\o_\#{}^{j}{}_0)+i\sum_{t=3}^{l-1} C_t,
\end{equation}
where $\{\cdot\}$ denotes the sum of all cyclic permutations in $\gamma_1,\ldots,\gamma_l$ and each $C_t$ is a sum of terms of the form
$$
g_{\mu_1\bar\nu}\omega_{\mu_2}{}^\#{}_{\mu_3;\mu_4\ldots\mu_t}
(\hat D_\#{}^{j}+\o_\#{}^{j}{}_0)_{;\mu_{t+1}\ldots\mu_l},
$$
where $\mu_1,\ldots,\mu_l$ is a permutation of $\gamma_1,\ldots,
\gamma_l$ such that $\mu_{t+1},\ldots,\mu_{l}$ are chosen from $\gamma_4,\ldots,\gamma_l$; also, in \eqref{e:636} the last sum is understood to be vacuous if $l=3$.

Let $z=(z^1,\ldots,z^n)\in \bC^n$ and, for each $2\leq l$, denote by $\Omega_{(l)}^\#(z)$ the homogeneous polynomial of degree $l$ obtained by multiplying $\omega_{\gamma_1}{}^\#{}_{\gamma_2;\gamma_3\ldots\gamma_l}$ by the monomials $z^{\gamma_1}\cdot\ldots\cdot z^{\gamma_l}$ (and summing according to the summation convention), i.e.\
\begin{equation}
\Omega_{(l)}^\#(z):=\omega_{\gamma_1}{}^\#{}_{\gamma_2;\gamma_3\ldots\gamma_l}
z^{\gamma_1}\ldots z^{\gamma_l},
\end{equation}
and define $\Omega^\#(z)$ to be the degree $l_0$ polynomial
\begin{equation}\Label{omega}
\Omega^\#(z):=\sum_{l=2}^{l_0}\Omega_{(l)}^\#(z),
\end{equation}
where $l_0$ is the integer in \eqref{k0deg}. We note that the $d=d_{l_0}$ polynomials $\Omega^\#(z)$, as $\#$ varies over its index set $\{n+1,\ldots, n+d\}$, are linearly independent in the space of polynomials $\bC[z]$ and, hence, spans a $d$-dimensional subspace of $\bC[z]$. (This follows from the fact that the rank of a matrix and that of its transpose are equal). Moreover, if we truncate the polynomials $\Omega^\#(z)$ at degree $l<l_0$, i.e.\ consider the polynomial
\begin{equation}\Label{lomega}
{}_{l}\Omega^\#(z):=\sum_{t=2}^{l}\Omega_{(t)}^\#(z),
\end{equation}
then we obtain $d_{l}$ linearly independent polynomials ${}_{l}\Omega^\#(z)$, for $\#=n+1,\ldots,n+d_{l}$, and ${}_{l}\Omega^\#(z)\equiv 0$, for $\#=n+d_{l}+1,\ldots,n+d$.

If we now multiply the first identity in \eqref{e:300} by $z^\alpha z^\beta\bar{z^\nu}$ (and sum according to the summation convention), then we obtain the polynomial identity
\begin{equation}\Label{polyeq:1}
\Omega_{(2)}^\#(z)\, \omega_\#{}^j\left(\bar z\right)=2i\hat D^j(z)\nl z\nr^2,
\end{equation}
where $\omega_\#{}^j(\bar z)$ and $\hat D^j(z)$ are the linear polynomials given by
\begin{equation}\Label{e:700}
\omega_\#{}^j(z):=\omega_\#{}^j{}_{\bar\nu} \bar z^\nu,\quad \hat D^j(z)=\hat D_\a{}^jz^\a=\hat D_\b{}^jz^\b.
\end{equation}
Similarly, multiplying \eqref{e:636} by $z^{\gamma_1}\ldots z^{\gamma_l}\bar z^\nu$ (and summing), we obtain
\begin{multline}\Label{polyeq:2}
\Omega_{(l)}^\#(z)\, \omega_\#{}^j\left(\bar z\right)=\\i\left(l\Omega_{(l-1)}^\#(z)(\hat D_\#{}^j+\o_\#{}^j{}_0)+\sum_{t=3}^{l-1}c_{lt}\Omega_{(t-1)}^\#(z)
(T_{(l-t)})_\#{}^j(z)
\right)\nl z\nl^2,
\end{multline}
where
$$
(T_{(l-t)})_\#{}^j(z):=(\hat D_\#{}^{j}+\o_\#{}^{j}{}_0)_{;\gamma_{t+1}\ldots\gamma_l}z^{\gamma_{t+1}}\ldots z^{\gamma_l}
$$
and the $c_{lt}$ are combinatorial integers.
We shall need the following lemma to analyze the equations \eqref{polyeq:2}:

\begin{Lem}\Label{newlemma1}
Let $z=(z^1,\ldots,z^n)$ be coordinates in $\bC^n$ with $n\geq 2$, and $p_1(z),\ldots,p_m(z)$ homogeneous polynomials of degree $d$ in $z$, linearly independent in the complex vector space $\bC[z]$. Let $S$ be the vector space of homogeneous polynomials, $r(z)$, of degree $d-1$ in $z$ such that there are linear polynomials $q_1(\bar z),\ldots,q_m(\bar z)$ in $\bar z$ (depending on $r(z)$) satisfying
\begin{equation}\Label{polyeqdeg2}
\sum_{j=1}^mp_j(z)q_j(\bar z)=r(z)\nl z\nr^2.
\end{equation}
If, for some $0\leq k\leq n-1$, we have
$$
m<\sum_{j=0}^k(n-j),
$$
then
$$
\dim S\leq k.
$$
Moreover, the polynomial $r(z)\equiv 0$ in \eqref {polyeqdeg2} if and only if the $\bC^m$-valued linear polynomial $q(z)=(q_1(z),\ldots, q_m(z))\equiv 0$. Consequently, the space of $q(z)$ such that there exists an $r(z)$ satisfying \eqref {polyeqdeg2} also has dimension $\dim S$.
\end{Lem}

\begin{Rem} {\rm The estimate for $\dim S$ in this lemma is sharp as is illustrated by the example in Section \ref{Exlem1} below.
}
\end{Rem}

Before proving Lemma \ref{newlemma1}, we shall make some preliminary reductions and observations. Let $p(z):=(p_1(z),\ldots,p_m(z))$ be as in Lemma \ref{newlemma1} and assume that  $q(\bar z):=(q_1(\bar z),\ldots,q_m(\bar z))$, $r(z)$ solve \eqref{polyeqdeg2}. We can write $q(\bar z)=Q\bar z^t$, where $Q$ is a constant $m\times n$-matrix (the space of such matrices will be denoted by $\bC^{m\times n}$) and the superscript $t$ denotes the transpose of a matrix, and express the identity \eqref{polyeqdeg2} in matrix form as follows
\begin{equation}
p(z)Q\bar z^t= r(z)z\bar z^t,
\end{equation}
which is equivalent to
\begin{equation}\Label{pQr}
p(z)Q=r(z)z.
\end{equation}
We note that $r(z)\equiv 0$ if and only if $Q=0$, since the polynomials $p_j(z)$ are linearly independent.  This proves the last statement in Lemma \ref{newlemma1}. Moreover, if $r\not\equiv 0$, then the matrix $Q$ must have rank $n$ since the mapping $z\mapsto r(z)z$ is clearly not contained in any proper subspace of $\bC^n$ (simply note that this mapping restricted to a complex line $L$ through 0 maps $L$ onto itself unless $L$ is contained in the zero locus of $r(z)$), and, hence, $m\geq n$. This proves the conclusion of Lemma \ref{newlemma1} for $k=0$; we also note that this is direct consequence of Lemma 3.2 in \cite{Huang99}.

For the proof of Lemma \ref{newlemma1}, we shall need the following preliminary result. We shall identify a matrix $Q\in \bC^{m\times n}$ with a linear mapping $Q\colon \bC^m\to \bC^n$ via $v\in \bC^m \mapsto vQ\in \bC^n$.

\begin{Lem}\Label{prenewlemma1} Let $p(z)=(p_1(z),\ldots,p_m(z))$ be as in Lemma $\ref{newlemma1}$. Assume that the pairs $\left(Q^{(1)},r^{(1)}(z)\right),\ldots, \left (Q^{(k)},r^{(k)}(z)\right)$ are linearly independent solutions to \eqref{pQr}, i.e.\ each pair $Q=Q_j\in \bC^{m\times n}$ and $r(z)=r^{(j)}(z)$ satisfies \eqref{pQr} and the collection $Q^{(1)},\ldots Q^{(k)}$ is linearly independent in $\bC^{m\times n}$ (or equivalently $r^{(1)}(z),\ldots, r^{(k)}(z)$ are linearly independent in the space of polynomials in $z$). Let
$$\kappa:=\dim \left (\ker Q^{(1)}\cap\ldots\cap \ker Q^{(k)}\right),
$$
and let (if $\kappa\geq 1$) $v_1,\ldots v_{\kappa}\in \bC^m$ be linearly independent vectors spanning $\ker Q^{(1)}\cap\ldots\cap \ker Q^{(k)}$. Then, there are linear $\bC^{m}$-valued polynomials $s^{(1)}(z),\ldots,s^{(k)}(z)$ and (if $\kappa\geq 1$) polynomials $h_1(z),\ldots, h_\kappa(z)$ such that
\begin{equation}\Label{prep}
p(z)=\sum_{j=1}^\kappa h_j(z)v_j+\sum_{i=1}^k r^{(i)}(z)s^{(i)}(z),
\end{equation}
where the first sum in \eqref{prep} is vacuous if $\kappa=0$.
\end{Lem}

\begin{proof} We shall prove Lemma \ref{prenewlemma1} by induction on $k$. Let us first assume that $k=1$. As noted above, if $Q:=Q^{(1)}\neq 0$ and $r(z):=r^{(1)}(z)\not\equiv 0$ solves \eqref{pQr}, then the rank of $Q$ equals $n$. Consequently, $Q$ has a left inverse $S\in \bC^{n\times m}$, i.e.\ $SQ$ equals the identity $n\times n$-matrix $I$. It follows that $$(p(z)-r(z)zS)Q=p(z)Q-r(z)zSQ=r(z)z-r(z)zI=0,$$ i.e.\ $p(z)-r(z)s(z)$, with $s(z):=zS$, takes all its values in $\ker Q$. This proves that $p(z)-r(z)s(z)=\sum_{j=1}^{\kappa} h_j(z) v_j$, where $\kappa:=\dim \ker Q=m-n$, $v_1,\ldots, v_{\kappa}\in \bC^m$ span $\ker Q$, and $h_1(z),\ldots, h_{\kappa}(z)$ are homogeneous polynomials of degree $d$, completing the proof of \eqref{prep} for $k=1$.

Next, assume that Lemma \ref{prenewlemma1} holds for all $k=1,\ldots, k_0$. Let $$\left(Q^{(1)},r^{(1)}(z)\right),\ldots, \left (Q^{(k_0+1)},r^{(k_0+1)}(z)\right)$$ be $k_0+1$ pairs of linearly independent solutions to \eqref{pQr}. By the induction hypothesis, we can express $p(z)$ in the form \eqref{prep} with $k=k_0$ and $\kappa=\kappa_0$, where $\kappa_0$ denotes the dimension of $V_0:=\ker Q^{(1)}\cap\ldots\cap \ker Q^{(k_0)}$; i.e.\
\begin{equation}\Label{prep0}
p(z)=\sum_{j=1}^{\kappa_0} h_j(z)v_j+\sum_{i=1}^{k_0} r^{(i)}(z)s^{(i)}(z),
\end{equation}
Now, let $\kappa_1\leq \kappa_0$ denote the dimension of $V_1:= \ker Q^{(1)}\cap\ldots\cap \ker Q^{(k_0+1)}$. After performing an invertible linear transformation of the $v_1,\ldots,v_{\kappa_0}$ if necessary, we may assume (without loss of generality) that $v_1,\ldots, v_{\kappa_1}$ span $V_1$. Equation \eqref{pQr} then reads
\begin{equation}\Label{e:705}
\begin{aligned}
r^{k_0+1}(z)z &=p(z)Q^{(k_0+1)}\\&=\left(\sum_{j=1}^{\kappa_1} h_j(z)v_j+\sum_{j=\kappa_1+1}^{\kappa_0} h_j(z)v_j+\sum_{i=1}^{k_0} r^{(i)}(z)s^{(i)}(z)\right ) Q^{(k_0+1)}\\
&= \left(\sum_{j=\kappa_1+1}^{\kappa_0} h_j(z)v_j+\sum_{i=1}^{k_0} r^{(i)}(z)s^{(i)}(z)\right)Q^{(k_0+1)},
\end{aligned}
\end{equation}
or, equivalently,
\begin{equation}\Label{e:706}
\begin{aligned}
\sum_{j=\kappa_1+1}^{\kappa_0} h_j(z)v_jQ^{(k_0+1)}=
r^{k_0+1}(z)z -\sum_{i=1}^{k_0} r^{(i)}(z)s^{(i)}(z)Q^{(k_0+1)}.
\end{aligned}
\end{equation}
By construction, the span of the vectors $v_{\kappa_1+1},\ldots,v_{\kappa_0}$ intersects $\ker Q^{(k_0+1)}$ only at the zero vector and, hence, if $R$ denotes the $(\kappa_0-\kappa_1)\times m$ matrix whose rows consist of $v_{\kappa_1+1},\ldots,v_{\kappa_0}$, then the linear mapping $RQ^{(k_0+1)}\colon \bC^{\kappa_0-\kappa_1}\to \bC^n$ is injective and therefore has a right inverse $T\in \bC^{n\times(\kappa_0-\kappa_1)}$. If we write $h(z):=(h_{\kappa_1+1}(z),\ldots,h_{\kappa_0})$, then we can write
$$
\sum_{j=\kappa_1+1}^{\kappa_0} h_j(z)v_j=h(z)R.
$$
Thus, if we multiply \eqref{e:706} from the right by $T$, then we obtain
\begin{equation}\Label{e:7065}
\begin{aligned}
h(z) & =h(z)RQ^{(k_0+1)}T=\sum_{j=\kappa_1+1}^{\kappa_0} h_j(z)v_jQ^{(k_0+1)}T &=\\
& = r^{k_0+1}(z)zT -\sum_{i=1}^{k_0} r^{(i)}(z)s^{(i)}(z)Q^{(k_0+1)}T.
\end{aligned}
\end{equation}
By substituting this expression into \eqref{prep0}, we conclude that \eqref{prep} also holds for $k=k_0+1$, which completes the induction and, hence, the proof of Lemma \ref{prenewlemma1}.
\end{proof}

\begin{proof}[Proof of Lemma $\ref{newlemma1}$]  To prove the conclusion of Lemma \ref{newlemma1}, it suffices to show that: {\it If there are $k+1\geq 1$ linearly independent pairs $\left(Q^{(1)},r^{(1)}(z)\right),\ldots, \left (Q^{(k+1)},r^{(k+1)}(z)\right)$ that solve \eqref{pQr}, then $m\geq \sum_{j=0}^{k}(n-j)$.} As mentioned above, this statement for $k=0$ is a direct consequence of Lemma 3.2 in \cite{Huang99}. We shall proceed by induction on $k$. Thus, let us fix $k_0\geq 1$ and assume that the statement holds for $k<k_0$. Suppose that there are $k_0+1$ linearly independent solutions $\left(Q^{(1)},r^{(1)}(z)\right),\ldots, \left (Q^{(k_0+1)},r^{(k_0+1)}(z)\right)$. For each $k\leq k_0+1$, let $\kappa_k$ be the dimension of the subspace $V_k:=\ker Q^{(1)}\cap\ldots\cap \ker Q^{(k)}$. Thus, we have $V_{k_0+1}\subset\ldots\subset V_1$ and $\kappa_{k_0+1}\leq\ldots\leq \kappa_1$. We can choose a basis $v_1,\ldots, v_{\kappa_1}$ for $V_1$ such that $v_1,\ldots, v_{\kappa_k}$ is a basis for $V_k$ (for every $k$ such that $\kappa_k\geq 1$). By Lemma \ref{prenewlemma1}, we can express $p(z)$, for each $k\leq k_0$, in the form
\begin{equation}\Label{e:708}
p(z)=\sum_{j=1}^{\kappa_k} h_j(z)v_j+\sum_{i=1}^{k} r^{(i)}(z)s_k^{(i)}(z),
\end{equation}
where the $s_k^{(i)}(z)$ depend on $k$ and the first sum is vacuous if $\kappa_k=0$.
Since
we also have $p(z)Q^{(k+1)}=r^{(k+1)}(z)z$, we conclude from \eqref{e:708} that
\begin{equation}\Label{e:7085}
\begin{aligned}
r^{(k+1)}(z)z &=\left(\sum_{j=1}^{\kappa_k} h_j(z)v_j+\sum_{i=1}^{k} r^{(i)}(z)s_k^{(i)}(z)\right)Q^{(k+1)}\\
&= \left(\sum_{j=\kappa_{k+1}+1}^{\kappa_k} h_j(z)v_j\right)Q^{(k+1)}+\sum_{i=1}^{k} r^{(i)}(z)\tilde s_k^{(i)}(z),
\end{aligned}
\end{equation}
where $\tilde s_k^{(i)}(z):=s_k^{(i)}(z)Q^{(k+1)}$.

Let $R_k\subset \bC^n$ denote the homogeneous algebraic variety, each component of which has dimension at least $n-k$, defined by
$$r^{(1)}(z)=\ldots =r^{(k)}(z)=0,$$
and suppose first that there is a component $C$ of $R_k$ such that the restriction $r(z)$ of $r^{(k+1)}(z)$ to $C$ does not vanish identically, i.e.\ $r:=r^{(k+1)}\big|_C\not \equiv 0$. From \eqref{e:7085}, we conclude that on $C$ we have
\begin{equation}\Label{e:7087}
\begin{aligned}
r(z)z = \left(\sum_{j=\kappa_{k+1}+1}^{\kappa_k} h_j(z)v_j\right)Q^{(k+1)}.
\end{aligned}
\end{equation}
Now, note that for each complex line $L\subset C$ through 0 such that $r\big|_L\not \equiv 0$, the mapping $z\to r(z)z$ sends $L$ onto itself. Hence, this map sends the homogeneous variety $C$, which has dimension at least $n-k$, onto a Zariski open subset of itself. The right hand side of \eqref{e:7087} maps into a subspace of dimension at most $\kappa_k-\kappa_{k+1}$ and, hence, we conclude
\begin{equation}\Label{e:709}
\kappa_k-\kappa_{k+1}\geq n-k.
\end{equation}
However, even though $r^{(1)}(z),\ldots, r^{(k+1)}(z)$ are linearly independent in $\bC[z]$, it could happen that $r^{(k+1)}(z)\equiv 0$ on $R_k$. The linear independence implies (by homogeneity) that $r^{(k+1)}$ does not belong to the ideal $I:=I(r^{(1)},\ldots,r^{(k)})$, whereas $r^{(k+1)}(z)\equiv 0$ on $R_k$, by the Nullstellen Satz, only implies that $r^{(k+1)}$ belongs to the radical $\sqrt{I}$. To deal with this complication, we shall need to use some facts from commutative algebra for which we refer the reader to \cite{Hormander90} (Chapter 7.7) and \cite{Sturmfels02} (Chapter 10). Let
$$
I=J_1\cap\ldots\cap J_t
$$
be the primary decomposition of the ideal $I:=I(r^{(1)},\ldots,r^{(k)})$. Since $r^{(k+1)}\not\in I$, there is a primary ideal $J=J_i$, for some $i\in \{1,\ldots, t\}$, such that $r^{(k+1)}\not\in J$. Let $\frak p$ denote the associated prime ideal, i.e.\ $\frak p=\sqrt{J}$, and $C_{\frak p}$ the irreducible, homogeneous zero locus of $\frak p$. (The situation above, where $r^{(k+1)}\not\equiv 0$ on a component $C=C_{\frak p}$ of $R_k$, corresponds to the one where we also have
$r^{(k+1)} \not\in \frak {p}$, which need not hold in general.) Let now $\frak N$ denote the space of Noetherian operators associated to the primary ideal $J$ (see e.g.\ Chapter 7.7 in \cite{Hormander90} or Chapter 10 in \cite{Sturmfels02}); i.e.\ $\frak N$ consists of the collection of partial differential operator $P=P(z,\partial)$ with polynomial coefficients (elements of the Weyl algebra),
$$
P(z, \partial):=\sum_{|\epsilon|\leq s} a_\epsilon(z)\partial^\epsilon, \quad a_\alpha\in \bC[z],\quad \epsilon\in \mathbb Z_+^n,\quad \partial^\epsilon:=\left(\frac{\partial}{\partial z_1}\right)^{\epsilon_1}\ldots \left(\frac{\partial}{\partial z_n}\right)^{\epsilon_n},
$$
such that
$$
(Pf)(z):=P(z,\partial)f(z)\equiv 0\,\, \text{{\rm on $C_\frak p$}}.
$$
The main result concerning $\frak N$ is the following characterization of the primary ideal $J$ (see Theorem 7.7.6 in \cite{Hormander90}):
$$f\in J\quad \iff (Pf)(z)\equiv 0\,\, \text{{\rm on $C_\frak p$}}, \, \forall P(z,\partial)\in \frak N.
$$
It is well known (and easy to see) that if $P\in \frak N$, then $[z_j,P]\in \frak N$ for $j=1,\ldots, n$, and, as a consequence, it follows that if $P\in \frak N$, then $P_{(\delta)}\in \frak N$ for all multi-indices $\delta\in \mathbb Z_+^n$, where $P_{(\delta)}=P_{(\delta)}(z,\partial)$ denotes the partial differential operator corresponding to the symbol
$$
P_{(\delta)}(z,\zeta):=\left(\frac{\partial}{\partial\zeta}\right)^\delta P(z,\zeta).
$$
Also, recall Leibnitz rule for differentiating a product,
\begin{equation}\Label{Leibnitz}
P(z,\partial)(uv)=\sum_{\delta\in \mathbb Z_+^n} \frac{1}{\delta!}(P_{(\delta)}(z,\partial) u)\frac{\partial^{|\delta|}v}{\partial z^\delta}.
\end{equation}
Now, since $r^{(k+1)}\not \in J$, there exist partial differential operators $P\in \frak N$ such that
\begin{equation}\Label{e:7010}
(Pr^{(k+1)})(z)\not\equiv 0\quad \text{{\rm on $C_{\frak p}$.}}
\end{equation}
Let us choose such a $P$ of minimal order (as a partial differential operator; i.e.\ with minimal $s$ where $s$ is the maximal order of a derivative appearing in $P=P(z,\partial)$). By applying $P$ to \eqref{e:7085} and using Leibnitz rule, we conclude that on $C_{\frak p}$ we have
\begin{equation}\Label{e:7011}
\begin{aligned}
(Pr^{(k+1)})(z)z = \left(\sum_{j=\kappa_{k+1}+1}^{\kappa_k} (Ph_j)(z)v_j\right)Q^{(k+1)},
\end{aligned}
\end{equation}
since $(P_{(\delta)}r^{(i)})(z) \equiv 0$ on $C_{\frak p}$ for $i\in \{1,\ldots k\}$ and all $\delta\in \mathbb Z_+^n$ (as $r^{(i)}\in J$ and $P_{(\delta)}\in \frak N$), and $(P_{(\delta)}r^{k+1})(z)\equiv 0$ on $C_{\frak p}$ for all $\delta\neq (0,\ldots, 0)$ (as $P\in \frak N$ was chosen to have minimal order). The same argument used to conclude \eqref{e:709} from \eqref{e:7087} above shows that \eqref{e:7011} implies \eqref{e:709} as well.

To conclude the proof of Lemma \ref{newlemma1}, recall that $\kappa_1=m-n$. Thus, by telescoping \eqref{e:709} we obtain
\begin{equation}
\begin{aligned}
m-n &=\kappa_1\geq (n-1)+\kappa_2\geq (n-1)+(n-2)+\kappa_3\geq\ldots\\
&\geq (n-1)+\ldots+(n-k_0)+\kappa_{k_0+1},
\end{aligned}
\end{equation}
which proves
$$
m\geq \sum_{j=0}^{k_0}(n-j),
$$
as desired. This completes the inductive step and, hence, the proof of Lemma \ref{newlemma1}.
\end{proof}

We shall now return to the equations \eqref{polyeq:2} for $l\geq 2$. We shall denote by $s_l^j(z)$ the homogeneous polynomial of degree $l-1$ appearing on the right in \eqref{polyeq:2} so that this equation reads
\begin{equation}\Label{polyeq:3}
\Omega^\#_{(l)}(z) \omega_\#{}^j(\bar z)=s_l^j(z)\nl z\nr^2;
\end{equation}
thus, we have, e.g., $s_2^j(z)=2i\hat D^j(z)$, where $\hat D^j(z)$ is defined in \eqref{e:700}. Recall that, for any $l\leq l_0$, we denote by ${}_{l}\Omega^\#(z)$ the sum of the polynomials $\Omega_{(t)}(z)$ for $t\leq l$ (see \eqref{lomega}), and the $d_{l}$ polynomials ${}_{l}\Omega^\#(z)$, for $\#=n+1,\ldots, n+l$, are linearly independent. Also, recall that the last $d_{l}-d_{l-1}$ (where we understand $d_1$ to be $0$) of these polynomials, ${}_{l}\Omega^\#(z)$ for $\#=n+d_{l-1}+1,\ldots, n+d_{l}$, are homogeneous of degree $l$ by \eqref{strongnorm} and, therefore, equal to $\Omega_{(l)}^\#(z)$. We conclude that $\Omega_{(l)}^\#(z)$, for $\#=n+d_{l-1}+1,\ldots, n+d_{l}$, are linearly independent. Now, we can break up the sum on the left in \eqref{polyeq:3} and rewrite this equation as follows:
\begin{equation}\Label{polyeq:4}
\sum_{\#=n+d_{l-1}+1}^{n+d_{l}}\Omega_{(l)}^\#(z)\omega_\#{}^j(\bar z)=
s_{l}^j(z)\nl z\nr^2-
\sum_{\#=n+1}^{n+d_{l-1}}\Omega_{(l)}^\#(z)\omega_\#{}^j(\bar z).
\end{equation}
For $l=2$, the sum on the right is vacuous and \eqref{polyeq:4} reduces to \eqref{polyeq:3}. Recall the following assumption from Theorem \ref{Main0}:
\medskip

\noindent
{\bf Assumption 1:} {\it There are integers $k_2,\ldots, k_{l_0}$ with $0\leq k_l\leq n-1$ such that the dimensions $d_l$ of $E_{l}(p)$ satisfy
\begin{equation}\Label{newass}
d_{l}-d_{l-1}<\sum_{j=0}^{k_l} (n-j),
\end{equation}
for $l=2,3,\ldots,l_0$ and where, for $l=2$ we understand $d_1=0$.
}
\medskip

\noindent
By applying Lemma \ref{newlemma1} to the equation \eqref{polyeq:4} with $l=2$, we conclude that there are at most $k_2$ linearly independent $\bC^{d_2}$-valued polynomials among the $(\omega_\#{}^j(\bar z))_{\#=n+1}^{n+d_2}$. Let $e_2\leq k_2$ denote the actual number of linearly independent $\bC^{d_2}$-valued polynomials among the $(\omega_\#{}^j(\bar z))_{\#=n+1}^{n+d_2}$. By moving to a nearby point $p_0\in M$ if necessary, we may assume that $e_2$ is locally constant near $p_0$. Hence, after a unitary transformation of the normal vector fields $L_{i}$ (with $i$ in its standard range $i=n+d+1,\ldots, N$), we may assume that $(\omega_\#{}^j(\bar z))_{\#=n+1}^{n+d_2}$ are linearly independent for $j=n+d+1,\ldots,n+d+e_2$ and $(\omega_\#{}^j(\bar z))_{\#=n+1}^{n+d_2}\equiv 0$ for $j\geq n+d+e_2+1$. Next, consider \eqref{polyeq:4} for $l=3$ and $j\geq n+d+e_2+1$. By the just accomplished normalization of $(\omega_\#{}^j(\bar z))_{\#=n+1}^{n+d_2}$,  we have
\begin{equation}\Label{polyeq:l=3}
\sum_{\#=n+d_{2}+1}^{n+d_{3}}\Omega_{(3)}^\#(z)\omega_\#{}^j(\bar z)=
s_{3}^j(z)\nl z\nr^2,\quad j\geq n+d+e_2+1.
\end{equation}
By Lemma \ref{newlemma1}, we conclude that there are at most $k_3$ linearly independent $\bC^{d_3-d_2}$-valued polynomials among the $(\omega_\#{}^j(\bar z))_{\#=n+d_2+1}^{n+d_3}$ for $j\geq n+d+e_2+1$. We let $e_3\leq k_3$ denote the actual number of linearly independent ones and, as above, we may assume that $e_3$ is locally constant near $p_0\in M$ and perform a unitary transformation among the $L_i$, now with $i=n+d+e_2+1,\ldots, N$, such that the $(\omega_\#{}^j(\bar z))_{\#=n+d_2+1}^{n+d_3}$ are linearly independent for $j=n+d+e_2+1,\ldots, n+d+e_2+e_3$ and zero for $j\geq n+d+e_2+e_3+1$. Proceeding inductively, we will accomplish the following normalization for each $l=2,\ldots, l_0$, with the understanding that $d_1=e_1=0$:
\begin{equation}\Label{o}
\left\{
\begin{aligned}
&(\omega_\#{}^j(\bar z))_{\#=n+d_{l-1}+1}^{n+d_l}\, \text{{\rm linearly independent for $j=n+d+e[l-1]+1,\ldots, n+d+e[l]$}},\\
&(\omega_\#{}^j(\bar z))_{\#=n+d_{l-1}+1}^{n+d_l}=0\, \text{{\rm for $j=n+d+e[l]+1,\ldots, N$}},
\end{aligned}
\right.
\end{equation}
where we have used the notation $$e[l]:=e_1+\ldots+e_l.$$
Moreover, the equations \eqref{polyeq:3}, for $l\geq 2$, reduce to
\begin{equation}\Label{polyeq:l}
\sum_{\#=n+d_{l-1}+1}^{n+d_{l}}\Omega_{(l)}^\#(z)\omega_\#{}^j(\bar z)=
s_{l}^j(z)\nl z\nr^2,\quad  j=n+d+e[l-1]+1,\ldots,n+d+e[l].
\end{equation}
Also, by Lemma \ref{newlemma1}, the polynomials $s^j_l(z)$ are all linearly independent for $2\leq l\leq l_0$ and $j=n+d+e[l-1]+1,\ldots,n+d+e[l]$. We shall now proceed under the following assumption, which clearly holds under the second assumption in \eqref{conds0} in Theorem \ref{Main0}:
\medskip

\noindent
{\bf Assumption 2:} {\it The integer
\begin{equation}\Label{newass2}
e:=e[l_0]=\sum_{l=2}^{l_0}e_l<n.
\end{equation}
}
\medskip

\noindent
We shall now introduce the following further conventions: For $l=2,\ldots,l_0$, the indices $i_l, j_l$ will run over the index set $\{n+d+e[l-1]+1,\ldots, n+d+e[l]\}$ (again with the understanding that $e_1=e[1]=0$) and the indices $i'_l, j'_l$ will run over the complementary set $\{n+d+e[l]+1,\ldots, N\}$. We shall also use the convention that $i_0,j_0$ run over the index set $\{n+d+1,\ldots, n+d+e\}$ and $i'_0,j'_0$ over the complementary set $\{n+d+e+1,\ldots, N\}$ (so that in fact e.g.\ $i'_0$ runs over the same index set as $i'_{l_0}$).
Recall that we have
\begin{equation}\Label{e:795}
\hat D_\a{}^{j'_2}=0,\quad \o_\#{}^{j'_l}{}_{\bar\nu}=0\quad \text{{\rm for $\#=n+1,\ldots, n+d_l$}},
\end{equation}
which implies, by \eqref{eq-phiia},
\begin{equation}\Label{e:800}
\hat\phi_\a{}^{j'_2}=0,\quad
\hat\phi^{j'_2}=\hat E^{j'_2}\theta
\end{equation}
and, by \eqref{e:610},
\begin{equation}\Label{e:801}
\hat\phi_\#{}^{j'_l}=(\hat D_\#{}^{j'_l}+\o_\#{}^{j'_l}{}_0)\theta\quad \text{{\rm for $\#=n+1,\ldots, n+d_l$}}.
\end{equation}
In particular, we have $\hat\phi_\#{}^{j'_{l_0}}=0\mod \theta$ for all $\#$ in its range.
Differentiating the identity in \eqref{e:801}, we obtain
\begin{equation}\Label{e:900}
d\hat\phi_\#{}^{j'_l}= d(\hat D_\#{}^{j'_l}+\o_\#{}^{j'_l}{}_0)\wedge\theta+ig_{\mu\bar\nu}(\hat D_\#{}^{j'_l}+\o_\#{}^{j'_l}{}_0)\theta^\mu\wedge\theta^{\bar\nu}\quad \text{{\rm for $\#=n+1,\ldots, n+d_l$}}
\end{equation}
and the corresponding structure equations for $\hat\phi_\#{}^{j'_l}$ reduce to (using the facts that $\theta_\#$ and $\phi_\mu{}^{j'_l}=0$ on $M$),
\begin{equation}\Label{e:1000}
\begin{aligned}
d\hat\phi_\#{}^{j'_l}=&
\hat\phi_\#{}^a\wedge\hat\phi_a{}^{j'_l}\\=&\hat\phi_\#{}^{*}\wedge\hat\phi_{*}{}^{j'_l} + \hat\phi_\#{}^{i_0}\wedge\hat\phi_{i_0}{}^{j'_l} + \hat\phi_{\#}{}^{i'_{0}}\wedge\hat\phi_{i'_{0}}{}^{j'_l}.
\end{aligned}
\end{equation}
Let us consider \eqref{e:900} and \eqref{e:1000} with $l=l_0$. If we identify the coefficients in front of $\theta^\mu\wedge\theta^{\bar\nu}$ on the right hand sides of \eqref{e:900} and \eqref{e:1000}, we obtain (by \eqref{e:801}, using also the identity \eqref{e:400})
\begin{equation}
ig_{\mu\bar\nu}(\hat D_\#{}^{j'_{0}}+\o_\#{}^{j'_{0}}{}_0)=- \o_\#{}^{i_0}{}_{\bar\nu}\o_{i_0}{}^{j'_{0}}{}_\mu.
\end{equation}
Multiplying both sides by $z^\mu \bar z^{\nu}$ and summing according to convention, we obtain (using notation analogous to that in \eqref{e:700})
\begin{equation}
i(\hat D_\#{}^{j'_{0}}+\o_\#{}^{j'_{0}}{}_0)\nl z \nr^2=-\o_\#{}^{i_0}(\bar z)\o_{i_0}{}^{j'_{0}}(z).
\end{equation}
Since $e<n$ by Assumption 2, Lemma 3.2 in \cite{Huang99} implies that $\hat D_\#{}^{j'_{0}}+\o_\#{}^{j'_{0}}{}_0=0$. Moreover, in view of \eqref{o} (which can be interpreted as saying that the polynomial $d\times e$ matrix $(\omega_\#{}^{i_{l_0}}(\bar z))$ has rank $e$ over $\bC$), we also conclude that $\o_{i_0}{}^{j'_{0}}(z)\equiv 0$, and we thus have established
\begin{equation}\Label{e:1100}
\hat D_\#{}^{j'_{0}}+\o_\#{}^{j'_{0}}{}_0=0,\quad \o_{i_0}{}^{j'_{0}}{}_\mu=0.
\end{equation}
Hence, in particular, we have $\hat\phi_\#{}^{j'_{0}}=0$. If we now differentiate the second identity in \eqref{e:800}, we obtain
$$
d\hat\phi^{j'_2}=d\hat E^{j'_2}\wedge\theta+ig_{\mu\bar\nu}\hat E^{j'_2}\theta^\mu\wedge\theta^{\bar\nu},
$$
while the structure equation for $\hat \phi^{j'_{0}}$ reads (in view of the vanishing of $\hat\phi_\a{}^{j'_2}$, and $\hat\phi_\#{}^{j'_{0}}$)
$$
d\hat\phi^{j'_{0}}=\hat\phi^{i_2}\wedge\hat\phi_{i_2}{}^{j'_{0}}+
\hat\phi^{i'_2}\wedge\hat\phi_{i'_2}{}^{j'_{0}}.
$$
By again identifying coefficients in front of $\theta^\mu\wedge\theta^{\bar\nu}$ (using $\hat\phi^{i'_2}=0$ mod $\theta$), we obtain
$$
ig_{\mu\bar\nu}\hat E^{j'_{0}}=\hat D_\mu{}^{i_2}\o_{i_2}{}^{j'_{0}}{}_{\bar\nu},
$$
which as a polynomial identity as above can be written
\begin{equation}\Label{e:1200}
i\hat E^{j'_{0}}\nl z \nr^2=\hat D^{i_2}(z)\tilde \o_{i_2}{}^{j'_{0}}(\bar z);
\end{equation}
here, we have used the notation $\tilde \o_{i_0}{}^{j'_{0}}(\bar z):=\o_{i_0}{}^{j'_{0}}{}_{\bar\nu}\bar z^\nu$ to distinguish this polynomial from that obtained by substituting $\bar z$ for $z$ in $\o_{i_0}{}^{j'_{0}}(z):=\o_{i_0}{}^{j'_{0}}{}_{\mu} z^\mu$, which we have already shown to be 0.
By Lemma 3.2 in \cite{Huang99} as above, we conclude that $\hat E^{j'_{0}}=0$ and, since  the $e_2$ polynomials $\hat D^{i_2}(z)$ are linearly independent (by \eqref{polyeq:3} with $l=2$ and \eqref{o}), we also deduce $\tilde \o_{i_2}{}^{j'_{0}}(\bar z)\equiv 0$, which is equivalent to $\o_{i_2}{}^{j'_{0}}{}_{\bar\nu}=0$. We may also write this (by using $\hat D_\mu{}^{i'_2}=0$) as
\begin{equation}\Label{e:2100}
\hat D_\mu{}^{i_0}\o_{i_0}{}^{j'_{0}}{}_{\bar\nu}=0.
\end{equation}
Since we have already established $\o_\#{}^j{}_\b=0$, $\o_{i_0}{}^{j'_{0}}{}_\b= 0$, it follows that covariant derivatives in the $\theta^\beta$ direction of the left hand side of \eqref{e:2100}  remain 0, i.e.\
\begin{equation}\Label{e:2200}
 \hat D_\mu{}^{i_0}{}_{;\b}\o_{i_0}{}^{j'_{0}}{}_{\bar\nu} +  \hat D_\mu{}^{i_0}\o_{i_0}{}^{j'_{0}}{}_{\bar\nu;\b}=0.
\end{equation}
Next, note that we have:
\begin{equation}\Label{e:2300}
\hat\phi_{i_0}{}^{j'_{0}}=\o_{i_0}{}^{j'_{0}}{}_{\bar\nu}\theta^{\bar\nu}+(\hat D_{i_0}{}^{j'_{0}}+\o_{i_0}{}^{j'_{0}}{}_0)\theta,
\end{equation}
which implies (via the structure equation for $\theta^{\bar\nu}$; see \eqref{psh2}) that
\begin{equation}\Label{e:2400}
d\hat\phi_{i_0}{}^{j'_{0}}=d\o_{i_0}{}^{j'_{0}}{}_{\bar\nu}\wedge \theta^{\bar\nu}-\o_{i_0}{}^{j'_{0}}{}_{\bar\gamma}\o_{\bar\nu}{}^{\bar\gamma}\wedge \theta^{\bar\nu}+ig_{\mu\bar\nu}(\hat D_{i_0}{}^{j'_{0}}+\o_{i_0}{}^{j'_{0}}{}_0)\theta^\mu\wedge\theta^{\bar\nu}\mod\theta
\end{equation}
The structure equation for $\hat\phi_{i_0}{}^{j'_0}$ can be expressed as
\begin{equation}\Label{e:2500}
\begin{aligned}
d\hat\phi_{i_0}{}^{j'_0}=&
\hat\phi_{i_0}{}^a\wedge\hat\phi_a{}^{j'_0}\mod\theta\\ =&
\hat\phi_{i_0}{}^{i}\wedge\hat\phi_{i}{}^{j'_0}\\ =&
\hat\phi_{i_0}{}^{j_{0}}\wedge\hat\phi_{j_{0}}{}^{j'_0}+ \hat\phi_{i_0}{}^{i'_0}\wedge\hat\phi_{i'_0}{}^{j'_0}\\
=&
\o_{j_0}{}^{j'_0}{}_{\bar\nu}\o_{i_0}{}^{j_0}\wedge \theta^{\bar\nu}-\o_{i_0}{}^{i'_0}{}_{\bar\nu}\o_{i'_0}{}^{j'_0}\wedge \theta^{\bar\nu} \mod\theta.
\end{aligned}
\end{equation}
We observe that the already established identities $\o_\#{}^{j}{}_{\mu}=0$, $\o_\#{}^{j'_0}{}_{\bar\nu}=0$ and
$\o_{j_0}{}^{j'_0}{}_\mu=0$ imply that
$$
\o_{j_0}{}^{j'_0}{}_{\bar\nu}\o_{i_0}{}^{j_0}{}_\mu=
\o_{a}{}^{j'_0}{}_{\bar\nu}\o_{i_0}{}^{a}{}_\mu,\quad
\o_{i_0}{}^{i'_0}{}_{\bar\nu}\o_{i'_0}{}^{j'_0}{}_\mu=\o_{i_0}{}^{a}{}_{\bar\nu}
\o_{a}{}^{j'_0}{}_\mu.
$$
Combining this observation with an identification of the coefficients in front of $\theta^\mu\wedge\theta^{\bar\nu}$ in \eqref{e:2400} and \eqref{e:2500} (and using the definition of the covariant derivatives) yields
\begin{equation}
\o_{i_0}{}^{j'_0}{}_{\bar\nu:\mu}+ig_{\mu\bar\nu}(\hat D_{i_0}{}^{j'_0}+\o_{i_0}{}^{j'_0}{}_0)=0.
\end{equation}
Equation \eqref{e:2200} implies
\begin{equation}\Label{e:2600}
 \hat D_\mu{}^{i_0}{}_{;\b}\o_{i_0}{}^{j'_0}{}_{\bar\nu}=  ig_{\beta\bar\nu}\hat D_\mu{}^{i_0}(\hat D_{i_0}{}^{j'_0}+\o_{i_0}{}^{j'_0}{}_0).
\end{equation}
Since the sum on the left has $e\leq n-1$ terms,
we conclude (by Lemma 3.2 in \cite{Huang99} as above) that
\begin{equation}\Label{e:2700}
 \hat D_\mu{}^{i_0}{}_{;\b}\o_{i_0}{}^{j'_0}{}_{\bar\nu}=0.
\end{equation}
Proceeding inductively, we obtain for any $l\geq 1$
\begin{equation}\Label{e:2800}
 \hat D_{\gamma_1}{}^{i_0}{}_{;\gamma_2\ldots\gamma_{l}}\o_{i_0}{}^{j'_0}{}_{\bar\nu}=0.
\end{equation}
Let us now return to the equations \eqref{e:900} and \eqref{e:1000} for general $l$ and $\#=n+1,\ldots, n+d_l$. If we identify the (wedge) multiples of $\theta$ on both sides, we obtain the following identity
\begin{multline}\Label{coveq0}
d(\hat D_\#{}^{j'_l}+\omega_\#{}^{j'_l}{}_{0})= \omega_\#{}^*(\hat D_*{}^{j'_l}+\omega_*{}^{j'_l}{}_{0})-\omega_*{}^{j'_l}(\hat D_\#{}^{*}+\omega_\#{}^{*}{}_{0})+\\
\omega_{\#}{}^{i_l}(\hat D_{i_l}{}^{j'_l}+\omega_{i_l}{}^{j'_l}{}_{0})-\omega_{i_l}{}^{j'_l}(\hat D_\#{}^{i_l}+\omega_\#{}^{i_l}{}_{0})+\\
\omega_{\#}{}^{i'_l}(\hat D_{i'_l}{}^{j'_l}+\omega_{i'_l}{}^{j'_l}{}_{0})-\omega_{i'_l}{}^{j'_l}(\hat D_\#{}^{i'_l}+\omega_\#{}^{i'_l}{}_{0})\, \mod \theta.
\end{multline}
Since the combined ranges of the indices $*,i_l,i'_l$ equal the range of the index $a$, we observe that \eqref{coveq0} simply states that the covariant derivatives
\begin{equation}\Label{coveq01}
(\hat D_\#{}^{j'_l}+\omega_\#{}^{j'_l}{}_{0})_{;\mu}=(\hat D_\#{}^{j'_l}+\omega_\#{}^{j'_l}{}_{0})_{;\bar\nu}=0,\quad
\text{{\rm for $\#=n+1,\ldots, n+d_l$}}.
\end{equation}
If we covariantly differentiate the third equation in \eqref{e:600} in the $\theta^\mu$ direction (again, the only components involved will be those appearing in this equation by \eqref{e:400}), then we obtain
\begin{equation}
\hat D_\a{}^j{}_{;\,\beta\mu} =\omega_\a{}^\#{}_{\beta;\mu}(\hat D_\#{}^j+\omega_\#{}^j{}_{0}) + \omega_\a{}^\#{}_{\beta}(\hat D_\#{}^{j}+\omega_\#{}^{j}{}_{0})_{;\mu}.
\end{equation}
If we now restrict to $j=j'_2$, then by \eqref{coveq01} we obtain
\begin{equation}
\hat D_\a{}^{j'_2}{}_{;\,\beta\mu} =\omega_\a{}^\#{}_{\beta;\mu}(\hat D_\#{}^{j'_2}+\omega_\#{}^{j'_2}{}_{0}) .
\end{equation}
Proceeding inductively, we obtain, for any $l\leq l_0$,
\begin{equation}\Label{e:2800'}
\hat D_{\gamma_1}{}^{j'_{l-1}}{}_{;\,\gamma_{2}\ldots\gamma_{l}} =\omega_{\gamma_1}{}^\#{}_{\gamma_{2};\gamma_3\ldots\gamma_{l}}(\hat D_\#{}^{j'_{l-1}}+\omega_\#{}^{j'_{l-1}}{}_{0}) .
\end{equation}
Note, in particular, that the range of the index $j_k$ belongs to the range of $j'_{l-1}$ for $k\geq l$.
If we now re-examine how the equations \eqref{e:636} are obtained inductively, we notice first that for $l=4$, by differentiating \eqref{omega3ref} with respect to $\theta^\gamma$, we obtain with $j=j_4$ (whose range belongs to that of $j'_2$)
\begin{equation}
\begin{aligned}
\omega_\a{}^\#{}_{\beta;\mu\gamma}\omega_\#{}^{j_4}{}_{\bar\nu} & = -\omega_\a{}^\#{}_{\beta;\mu}\omega_\#{}^{j_4}{}_{\bar\nu;\gamma}+ i (g_{\a\bar\nu}\hat D_\b{}^{j_4}{}_{;\mu\gamma}+g_{\b\bar\nu}\hat D_\mu{}^{j_4}{}_{;\a\gamma}
+g_{\mu\bar\nu}\hat D_\a{}^{j_4}{}_{;\b\gamma})\\
&= ig_{\gamma\bar\nu}\omega_\a{}^\#{}_{\beta;\mu}(\hat D_\#{}^{j_4}+\omega_\#{}^{j_4}{}_{0})+ i (g_{\a\bar\nu}\hat D_\b{}^{j_4}{}_{;\mu\gamma}+g_{\b\bar\nu}\hat D_\mu{}^{j_4}{}_{;\a\gamma}
+g_{\mu\bar\nu}\hat D_\a{}^{j_4}{}_{;\b\gamma})\\
&= ig_{\gamma\bar\nu}\hat D_{\alpha}{}^{j_4}{}_{;\beta\mu}+ i (g_{\a\bar\nu}\hat D_\b{}^{j_4}{}_{;\mu\gamma}+g_{\b\bar\nu}\hat D_\mu{}^{j_4}{}_{;\a\gamma}
+g_{\mu\bar\nu}\hat D_\a{}^{j_4}{}_{;\b\gamma})\\
&=  i (g_{\a\bar\nu}\hat D_\b{}^{j_4}{}_{;\mu\gamma}+g_{\b\bar\nu}\hat D_\mu{}^{j_4}{}_{;\a\gamma}
+g_{\mu\bar\nu}\hat D_\a{}^{j_4}{}_{;\b\gamma}+g_{\gamma\bar\nu}\hat D_{\alpha}{}^{j_4}{}_{;\beta\mu}),
\end{aligned}
\end{equation}
where in the second line we have used \eqref{e:640} and in the third line \eqref{e:2800'} with $l=3$ (or, equivalently, the identity just above). Proceeding inductively, we conclude that for all $l\geq 2$
\begin{equation}\Label{636'}
\omega_{\gamma_1}{}^\#{}_{\gamma_2;\gamma_3\ldots\gamma_l}\omega_\#{}^{j_l}{}_{\bar\nu}=
i \{g_{\gamma_1\bar\nu}\hat D_{\gamma_2}{}^{j_l}{}_{;\gamma_3\ldots\gamma_l}\}
\end{equation}
where $\{\cdot\}$ denotes the sum of all cyclic permutations in $\gamma_1,\ldots,\gamma_l$. It follows that \eqref{polyeq:2} can be written
\begin{equation}\Label{polyeq:l'}
\Omega^\#_{(l)}\omega_\#{}^{j_l}(\bar z)=
\sum_{\#=n+d_{l-1}+1}^{n+d_{l}}\Omega_{(l)}^\#(z)\omega_\#{}^{j_l}(\bar z)=
liD^{j_l}_{(l)}(z)\nl z\nr^2,
\end{equation}
where
\begin{equation}
D^{j_l}_{(l)}(z):=\hat D_{\gamma_2}{}^{j_l}{}_{;\gamma_3\ldots\gamma_l}z^{\gamma_1}\ldots z^{\gamma_l}
\end{equation}
and, hence, we conclude that $s^{j_l}_l(z)=D^{j_l}_{(l)}(z)$ for each $l$. The fact that these $e$ polynomials are linearly independent (by Lemma \ref{newlemma1} and the definition of the $e_l$), together with \eqref{e:2800}, now implies that
\begin{equation}
\omega_{i_0}{}^{j'_0}{}_{\bar\nu}=0.
\end{equation}
Thus, to summarize, we now have
\begin{equation}\Label{e:2900}
\hat \phi_\a{}^{j'_0}=0,\quad \hat\phi^{j'_0}=0,\quad \hat\phi_\#{}^{j'_0}=0,\end{equation}
and
$$
\hat\phi_{i_0}{}^{j'_0}=(\hat D_{i_0}{}^{j'_0}+\o_{i_0}{}^{j'_0}{}_0)\theta.
$$
Differentiating the latter identity yields
$$
d\hat\phi_{i_0}{}^{j'_0}=d(\hat D_{i_0}{}^{j'_0}+\o_{i_0}{}^{j'_0}{}_0)\wedge \theta+ig_{\mu\bar\nu}(\hat D_{i_0}{}^{j'_0}+\o_{i_0}{}^{j'_0}{}_0)\theta^\mu\wedge\theta^{\bar\nu},
$$
while the structure equation for $\hat\phi_{i_0}{}^{j'_0}$ reads
$$
d\hat\phi_{i_0}{}^{j'_0}=\hat\phi_{i_0}{}^{j_0}\wedge\hat\phi_{j_0}{}^{j'_0}+
\hat\phi_{i_0}{}^{i'_0}\wedge\hat\phi_{i'_0}{}^{j'_0}=0 \mod \theta.
$$
It follows (by considering the coefficient in front of $\theta^\mu\wedge\theta^{\bar\nu}$) that $\hat D_{i_0}{}^{j'}+\o_{i_0}{}^{j'}{}_0=0$, and hence
\begin{equation}\Label{e:3000}
\hat\phi_{i_0}{}^{j'}=0.
\end{equation}
The identities \eqref{e:2900} and \eqref{e:3000}, together with the adapted Q-frames in Section 8 of \cite{EHZ04} and the arguments (more or less verbatim) in the last paragraph of the proof of Theorem 2.2 in \cite{EHZ04}, now show that $f(M)$ is contained in a complex plane of dimension $n+d+e+1$. Since $e\leq k$ by definition, we have proved the conclusion of Theorem \ref{Main0}. \qed

\section{An Example illustrating Lemma $\ref{newlemma1}$}\Label{Exlem1}

In this section, we shall give a prototypical example showing that the estimate for $\dim S$ in Lemma $\ref{newlemma1}$ is sharp. Let
\begin{equation}\Label{monoorder}
z_1^2,z_1z_2,\ldots, z_1z_{n},z_2^2,z_2z_3,\ldots, z_2z_{n},z_3^3,z_3z_4,\ldots,z_3z_{n},\ldots, z_{n}^2
\end{equation}
be an ordering of the $n(n+1)/2$ monomials of degree 2 in $z=(z_1,\ldots, z_n)$. For $1\leq k\leq n-1$, let
$$
m:=\sum_{j=0}^k(n-j)
$$
and let $p(z)=(p_1(z),\ldots, p_m(z))$ be the $\bC^m$-valued polynomial consisting of the first $m$ monomials in \eqref{monoorder}, i.e.\ $p_1(z)=z_1^2$ and $p_m(z)=z_kz_{n}$. Let $q^1(\bar z)=(q^1_1(\bar z),\ldots, q^1_{m}(\bar z))$ denote the $\bC^m$-valued linear polynomial in $\bar z$ given by $q^1(z)=(\bar z_1,\ldots,\bar z_{n},0,\ldots, 0)$. We note that
$$
\sum_{j=1}^mp_j(z)q^1_j(\bar z)=z_1\nl z\nr^2.
$$
Next, if $k\geq 2$, we let $q^2(\bar z)$ be given by $q^2(\bar z):=(0,\bar z_1,0,\ldots, \bar z_2,\bar z_3,\ldots,\bar z_{n},0,\ldots, 0)$, where $\bar z_2$ appears as the $(n+1)$th component (corresponding to the component $z_2^2$ in $p(z)$). We obtain
$$
\sum_{j=1}^mp_j(z)q^2_j(\bar z)=z_2\nl z\nr^2.
$$
Clearly, $q^1(\bar z)$ and $q^2(\bar z)$ are linearly independent. If $k\geq 3$, then we can define another linearly independent solution $q^3(\bar z)$ as follows, $$q^3(\bar z):=(0,0,\bar z_1,0,\ldots,0,\bar z_2,0,\ldots, 0,\bar z_3,\ldots, \bar z_n,0,\ldots, 0),$$ where $\bar z_2$ appears as the $(n+2)$th component (corresponding to the component $z_2z_3$ in $p(z)$) and $\bar z_3$ as the $(2n)$th component (corresponding to the component $z_3^2$ in $p(z)$). We get
$$
\sum_{j=1}^mp_j(z)q^3j(\bar z)=z_3\nl z\nr^2.
$$
We are confident that the reader recognizes the pattern and realizes that we can always construct $k$ linearly independent solutions $q^1(\bar z),\ldots, q^k(\bar z)$ to \eqref{polyeqdeg2}. Lemma \ref{newlemma1} asserts that we cannot construct more than this. 

\def\cprime{$'$}


\begin{thebibliography}{BEH11}

\bibitem[BEH08]{BEH08}
M.~Salah Baouendi, Peter Ebenfelt, and Xiaojun Huang.
\newblock Super-rigidity for {CR} embeddings of real hypersurfaces into
  hyperquadrics.
\newblock {\em Adv. Math.}, 219(5):1427--1445, 2008.

\bibitem[BEH11]{BEH09}
M.~Salah Baouendi, Peter Ebenfelt, and Xiaojun Huang.
\newblock Holomorphic mappings between hyperquadrics with small signature
  difference.
\newblock {\em Amer. J. Math.}, 133(6):1633--1661, 2011.

\bibitem[BH05]{BH05}
M.~Salah Baouendi and XiaoJun Huang.
\newblock Super-rigidity for holomorphic mappings between hyperquadrics with
  positive signature.
\newblock {\em J. Differential Geom.}, 69:379--398, 2005.

\bibitem[Car32]{Cartan32}
{\'E}lie Cartan.
\newblock Sur la g{\'e}om{\'e}trie pseudo-conforme des hypersurfaces de
  l'espace de deux variables complexes {II}.
\newblock {\em Ann. Scuola Norm. Sup. Pisa Cl. Sci. (2)}, 1(4):333--354, 1932.

\bibitem[Car33]{Cartan33}
{\'E}lie Cartan.
\newblock Sur la g{\'e}om{\'e}trie pseudo-conforme des hypersurfaces de
  l'espace de deux variables complexes.
\newblock {\em Ann. Mat. Pura Appl.}, 11(1):17--90, 1933.

\bibitem[CM74]{CM74}
S.~S. Chern and J.~K. Moser.
\newblock Real hypersurfaces in complex manifolds.
\newblock {\em Acta Math.}, 133:219--271, 1974.

\bibitem[D'A88]{DAngelo88}
John~P. D'Angelo.
\newblock Proper holomorphic maps between balls of different dimensions.
\newblock {\em Michigan Math. J.}, 35(1):83--90, 1988.

\bibitem[D'A91]{Dangelo91}
J.~P. D'Angelo.
\newblock Polynomial proper holomorphic mappings between balls, ii.
\newblock {\em Michigan Math. J.}, 38, 1991.

\bibitem[DL11]{JPDLebl11}
John~P. D'Angelo and Ji{\v{r}}{\'{\i}} Lebl.
\newblock Hermitian symmetric polynomials and {CR} complexity.
\newblock {\em J. Geom. Anal.}, 21(3):599--619, 2011.

\bibitem[DLP07]{JPDLeblPeters07}
John~P. D'Angelo, Ji{\v{r}}{\'{\i}} Lebl, and Han Peters.
\newblock Degree estimates for polynomials constant on a hyperplane.
\newblock {\em Michigan Math. J.}, 55(3):693--713, 2007.

\bibitem[EHZ04]{EHZ04}
Peter Ebenfelt, Xiaojun Huang, and Dmitry Zaitsev.
\newblock Rigidity of {CR}-immersions into spheres.
\newblock {\em Comm. Anal. Geom.}, 12(3):631--670, 2004.

\bibitem[EHZ05]{EHZ05}
Peter Ebenfelt, Xiaojun Huang, and Dmitry Zaitsev.
\newblock The equivalence problem and rigidity for hypersurfaces embedded into
  hyperquadrics.
\newblock {\em Amer. J. Math.}, 127(1):169--191, 2005.

\bibitem[ESar]{ESh10}
Peter Ebenfelt and Ravi Shroff.
\newblock Partial rigidity of {CR} embeddings of real hypersurfaces into
  hyperquadrics with small signature difference.
\newblock {\em Comm. Anal. Geom.}, to appear.

\bibitem[Far86]{Faran86}
James~J. Faran.
\newblock The linearity of proper holomorphic maps between balls in the low
  codimension case.
\newblock {\em J. Differential Geom.}, 24(1):15--17, 1986.

\bibitem[For86]{Forstneric86}
Franc Forstneri{\v{c}}.
\newblock Embedding strictly pseudoconvex domains into balls.
\newblock {\em Trans. Amer. Math. Soc.}, 295(1):347--368, 1986.

\bibitem[HJ01]{HuangJi01}
Xiaojun Huang and Shanyu Ji.
\newblock Mapping $\mathbb{B}^n$ into $\mathbb{B}^{2n-1}$.
\newblock {\em Inventiones Mathematicae}, 145:219--250, 2001.
\newblock 10.1007/s002220100140.

\bibitem[HJX06]{HuangJiXu06}
Xiaojun Huang, Shanyu Ji, and Dekang Xu.
\newblock A new gap phenomenon for proper holomorphic mappings from {$B^n$}
  into {$B^N$}.
\newblock {\em Math. Res. Lett.}, 13(4):515--529, 2006.

\bibitem[HJY09]{HuangJiYin09}
Xiaojun Huang, Shanyu Ji, and Wanke Yin.
\newblock Recent progress on two problems in several complex variables.
\newblock {\em Proceedings of the ICCM 2007, Internatioanl Press}, Vol
  I:563--575, 2009.

\bibitem[HJY12]{HuangJiYin12}
Xiaojun Huang, Shanyu Ji, and Wanke Yin.
\newblock On the third gap for proper holomorphic maps between balls.
\newblock {\em Preprint; http://front.math.ucdavis.edu/1201.6440}, 2012.

\bibitem[H{\"o}r90]{Hormander90}
Lars H{\"o}rmander.
\newblock {\em An introduction to complex analysis in several variables},
  volume~7 of {\em North-Holland Mathematical Library}.
\newblock North-Holland Publishing Co., Amsterdam, third edition, 1990.

\bibitem[Hua99]{Huang99}
Xiaojun Huang.
\newblock On a linearity problem for proper holomorphic maps between balls in
  complex spaces of different dimensions.
\newblock {\em J. Differential Geom.}, 51:13--33, 1999.

\bibitem[Lam01]{Lam01}
Bernhard Lamel.
\newblock Holomorphic maps of real submanifolds in complex spaces of different
  dimensions.
\newblock {\em Pacific J. Math}, 201(2):357--387, 2001.

\bibitem[Stu02]{Sturmfels02}
Bernd Sturmfels.
\newblock {\em Solving systems of polynomial equations}, volume~97 of {\em CBMS
  Regional Conference Series in Mathematics}.
\newblock Published for the Conference Board of the Mathematical Sciences,
  Washington, DC, 2002.

\bibitem[Tan62]{Tanaka62}
Noboru Tanaka.
\newblock On pseudo-conformal geometry of hypersurfaces of the space of $n$
  complex variables.
\newblock {\em J. Math. Soc. Japan}, 14:397--429, 1962.

\bibitem[Tan75]{Tanaka75}
Noboru Tanaka.
\newblock {\em A differential geometric study on strongly pseudo-convex
  manifolds}.
\newblock Lectures in Mathematics, Department of Mathematics, Kyoto University,
  No. 9. Kinokuniya Book-Store Co., Ltd., Tokyo, 1975.

\bibitem[Web78]{Webster78}
S.~M. Webster.
\newblock Pseudo-{H}ermitian structures on a real hypersurface.
\newblock {\em J. Differential Geom.}, 13(1):25--41, 1978.

\bibitem[Web79]{Webster79}
S.~M. Webster.
\newblock The rigidity of {C}-{R} hypersurfaces in a sphere.
\newblock {\em Indiana Univ. Math. J.}, 28(3):405--416, 1979.

\end{thebibliography}
\end{document}